\documentclass[12pt]{article}
\usepackage{amsmath,amssymb,amsthm, amsfonts}
\usepackage{etaremune, enumerate, float, verbatim}
\usepackage{algorithm}
\usepackage{algorithmic}
\usepackage{hyperref}
\usepackage{graphicx}
\usepackage{url}
\usepackage[mathlines]{lineno}
\usepackage{dsfont} % needed for double-struck 1
\usepackage{tikz, graphicx, subcaption, caption}% added CR 12/6/17
\usepackage[algo2e,ruled,vlined]{algorithm2e}
\usepackage{mathrsfs}
\usepackage{color}
\definecolor{red}{rgb}{1,0,0}

\definecolor{blue}{rgb}{0,0,.7}

\definecolor{green}{rgb}{0,.6,0}

\definecolor{purp}{rgb}{.5,0,.5}

\numberwithin{figure}{section}   % added LH 11/15/17

\newtheorem{thm}{Theorem}[section]
\newtheorem{cor}[thm]{Corollary}
\newtheorem{lem}[thm]{Lemma}
\newtheorem{prop}[thm]{Proposition}

\newtheorem{obs}[thm]{Observation}

\theoremstyle{definition}
\newtheorem{rem}[thm]{Remark}

\theoremstyle{definition}
\newtheorem{defn}[thm]{Definition}

\theoremstyle{definition}
\newtheorem{ex}[thm]{Example}

\newcommand{\diam}{\operatorname{diam}}
\newcommand{\dist}{\operatorname{dist}}
\newcommand{\ol}{\overline}

\newcommand{\bit}{\begin{itemize}}
\newcommand{\eit}{\end{itemize}}
\newcommand{\ben}{\begin{enumerate}}
\newcommand{\een}{\end{enumerate}}
\newcommand{\beq}{\begin{equation}}
\newcommand{\eeq}{\end{equation}}
\newcommand{\bea}{\begin{eqnarray*}} % * means no number
\newcommand{\eea}{\end{eqnarray*}}
\newcommand{\bpf}{\begin{proof}}
\newcommand{\epf}{\end{proof}\ms}
\newcommand{\bmt}{\begin{bmatrix}}
\newcommand{\emt}{\end{bmatrix}}
\newcommand{\ms}{\medskip}

\newcommand{\noi}{\noindent}

\newcommand{\Z}{\operatorname{Z}}
\newcommand{\Zp}{\operatorname{Z_+}}

\newcommand{\clf}{\mathscr{F}}
\newcommand{\pd}{\gamma_P} 
\newcommand{\F}{{\mathcal F}}
\newcommand{\zbar}{\ol\Z}

\newcommand{\xbar}{\ol X}
\newcommand{\xtar}{\mathscr{X}^{{\tiny\rm TAR}}}

\newcommand{\ztar}{\mathscr{Z}^{{\tiny\rm TAR}}}
\newcommand{\dtar}{\mathscr{D}^{{\tiny\rm TAR}}}

\newcommand{\ulzo}{\underline{z_0}}
\newcommand{\ulxo}{\underline{x_0}}
\newcommand{\xxo}{x_0}

\newcommand{\zzo}{z_0}

\newcommand{\psdztar}{\mathscr{Z}_+^{{\tiny\rm TAR}}}

\title{Isomorphisms  and  properties of TAR reconfiguration graphs for zero forcing and other $X$-set parameters}
\author{Novi H. Bong\thanks{Department of Mathematical Sciences, University of Delaware, Newark, DE 19716, USA (nhbong@udel.edu)}\and Joshua Carlson\thanks{Department of Mathematics and Computer Science, Drake University, Des Moines, IA 50311, USA (joshua.carlson@drake.edu)}\and Bryan Curtis\thanks{Department of Mathematics, Iowa State University,
Ames, IA 50011, USA (bcurtis1@iastate.edu)}\and Ruth Haas\thanks{Department of Mathematics, University of Hawaii at
M\=anoa, Honolulu, HI 96822, USA (rhaas@hawaii.edu)}\and Leslie Hogben\thanks{Department of Mathematics, Iowa State University,
Ames, IA 50011, USA and American Institute of Mathematics, 600 E. Brokaw Road, San Jos\'e, CA 95112, USA
(hogben@aimath.org).}}

\begin{document}
%\linenumbers
\maketitle \vspace{-20pt}

\begin{abstract} 
An $X$-TAR (token addition/removal) reconfiguration graph has as its vertices sets that satisfy some property $X$, with an edge between two sets  if one is obtained from the other by adding or removing one element. This paper considers the $X$-TAR graph for $X-$ sets  of vertices 
of a base graph $G$ where the $X$-sets of $G$ must satisfy certain conditions. Dominating sets, power dominating sets, zero forcing sets, and positive  semidefinite zero forcing sets are all examples of $X$-sets. 
  For graphs $G$ and $G'$ with no isolated vertices, it is shown that $G$ and $G'$ have isomorphic $X$-TAR reconfiguration graphs if and only if there is a  relabeling of the vertices of $G'$ such that $G$ and $G'$ have exactly the same $X$-sets.  The concept of an $X$-irrelevant vertex is introduced to facilitate analysis of $X$-TAR graph isomorphisms.
  Furthermore, results related to the connectedness of the zero forcing TAR graph are given. We present  families of graphs that   exceed known lower bounds for connectedness parameters.  
  \end{abstract}

\noi {\bf Keywords} reconfiguration;  token addition and removal; TAR; vertex; $X$-set;  isomorphism; zero forcing

\noi{\bf AMS subject classification} 68R10, 05C50, 05C57, 05C60, 05C69 %\medskip
%vspace{-10pt}

%%%%%%%%%%%%%%%%%%%%%%%%%%%%%%%%%%%%%%%%%%%

%==========================
%
% Section
%
%==========================

\section{Introduction}

The study of reconfiguration examines relationships among solutions to a problem. These solutions are modeled as  vertices in a graph called  a \textit{reconfiguration graph}.  A \textit{reconfiguration rule} describes the adjacency relationship in  a reconfiguration graph. When the solutions are sets, the  \emph{token addition or removal rule}  says that two sets $S$ and $S'$ are adjacent in the \emph{TAR reconfiguration graph} if $S'$ can be obtained from $S$ by adding or removing one element. For our purposes, solutions are sets of vertices of a graph $G$ defined by a property, such as being a dominating set, power dominating set, or zero forcing set of $G$, although reconfiguration is studied more broadly. The graph $G$ is called the \emph{base graph}. 
 In \cite{PD-recon}, Bjorkman et al.~established  various properties of a TAR reconfiguration graph in a universal framework for parameters that we call $X$-set parameters (see Definition \ref{X-set-param}), and we call the associated TAR reconfiguration graphs $X$-TAR graphs.    Examples of $X$-set parameters include the domination number, the power domination number,   the zero forcing number, and the positive semidefinite zero forcing number (definitions of these parameters are presented later in this introduction). 

  It is established in Theorems \ref{isomorph-isometry-u} and \ref{isomorph-same-zfs} that if  $G$ and $G'$ are graphs with no isolated vertices that have isomorphic $X$-TAR reconfiguration graphs, then there is a relabeling of the vertices of $G'$ such that $G$ and $G'$ have exactly the same $X$-sets. To prove Theorem \ref{isomorph-same-zfs} we introduce the concept of a set of irrelevant vertices (or an $X$-irrelevant set). 
 In Section \ref{ss:irrelevant}, $X$-irrelevant sets are used to  characterize the automorphism group of an $X$-TAR graph.  There we also discuss the existence (or nonexistence) of nonempty $X$-irrelevant sets when $X$ is the domination number, power domination number, zero forcing  number,  or  positive semidefinite zero forcing  number.  The study of which $X$-TAR graphs are unique (up to isomorphism of the base graph and with isolated vertices prohibited) is  parameter-specific, and we examine this question for zero forcing TAR graphs in Section \ref{ss:unique}.

 Connectedness, that is, whether any one feasible solution can be transformed into any other by the reconfiguration rule, is a fundamental question in the study of reconfiguration. Connectedness properties seem to be very parameter-specific.  In Section \ref{s:ZTAR} we present results related to connectedness for  zero forcing TAR reconfiguration graphs.  This  includes exhibiting families of graphs that exceed known bounds for parameters describing connectedness of the zero forcing TAR reconfiguration graphs.

 All graphs are simple, undirected, finite, and have nonempty vertex sets.  Let $G=(V(G),E(G))$ be a graph. The \emph{neighborhood} of a vertex $v$ of $G$ is the set of all vertices adjacent to $v$ and is denoted by $N_G(v)$. 
The \emph{closed neighborhood} of $v$ is $N_G[v]=N_G(v)\cup\{v\}$.  If $S$ is a set of vertices of $G$, then $N_G[S]=\cup_{x\in S} N_G[x]$.
 Let $G$ be a graph and color all the vertices of $G$ blue or white. Zero forcing is a process that changes the color of white vertices to blue by applying the \emph{standard color change rule}, i.e.,  a blue vertex $v$  can force a white vertex $w$ to change color to  blue if $w$ is the only white vertex in the neighborhood of $v$. A \textit{zero forcing set} is a set $S\subseteq V(G)$ that  can  result in all the vertices of $G$  turning blue by repeated application of the standard color change rule starting with exactly the vertices in $S$ blue. 
 In the context of TAR reconfiguration, ``solutions'' to zero forcing on a graph $G$ are zero forcing sets.
 The \emph{zero forcing number} of $G$, denoted by $\Z(G)$, is the minimum cardinality of a zero forcing set. The zero forcing TAR reconfiguration graph allows us to consider  all the zero forcing sets (of various sizes) and the relationships among them for a particular  graph.  The study of zero forcing reconfiguration was initiated in \cite{GHH}, where  the token exchange  zero forcing reconfiguration graph was studied. Note that for the zero forcing token exchange reconfiguration graph, only minimum zero forcing sets are used as vertices and two such sets are adjacent if one can be obtained from the other by exchanging one vertex.

Positive semidefinite zero forcing is a variant of zero forcing defined by the \emph{PSD color change rule}: Let $S$ be the set of blue vertices and let $W_1,\dots, W_k$ be the sets of vertices of the $k\ge 1$ components of $G- S$ (the graph obtained from  $G$ by deleting the vertices in  $S$).  If $u\in  S$, $w\in W_i$, and $w$ is the only white neighbor of  $u$ in $G[W_i\cup  S]$, then change the color of $w$ to blue (where $G[U]$ denotes the subgraph of $G$ induced by the vertices of $U$).  
A \emph{PSD zero forcing set} is a set $S\subseteq V(G)$ that  can  result in all the vertices of $G$  turning blue by repeated application of the PSD color change rule starting with exactly the vertices in $S$ blue. 
The \emph{PSD zero forcing number} of $G$, denoted by $\Zp(G)$, is the minimum cardinality of a  {PSD zero forcing set}. More information on zero forcing and positive semidefinite zero forcing can be found in \cite{HLA2}.

 A set $S$ of vertices of a graph $G$ is a \emph{dominating set} of $G$ if $N_G[S]=V(G)$ and the \emph{domination number} $\gamma(G)$ of $G$ is the minimum  cardinality of a dominating set of $G$.  A set $S$ of vertices of a graph $G$ is a \emph{power dominating set} of $G$ if $N_G[S]$ is a zero forcing set of $G$, and the \emph{power domination number} $\pd(G)$ of $G$ is the minimum cardinality of a power dominating set of $G$.  The study of $\gamma$-TAR reconfiguration was introduced in \cite{HS14} and $\pd$-TAR reconfiguration was introduced in \cite{PD-recon}.

In the next two sections we present the definitions of an $X$-set  parameter, its TAR graph, and related parameters, followed by basic results concerning the zero forcing TAR graph, including determinations of the zero forcing TAR graphs for several graph families.  

%==========================
\subsection{$X$-set parameter definitions}\label{ss:univ}

The next  definition is adapted from \cite{PD-recon} (the term {\em $X$-set parameter} is new but the list of conditions in Definition \ref{X-set-param} is taken from \cite[Definition 2.1]{PD-recon}).

\begin{defn}\label{X-set-param}
 An {\em $X$-set parameter} is a graph parameter $X(G)$ defined to be the minimum cardinality of an $X$-set of $G$, where the $X$-sets of $G$  are defined by a given property and satisfy the following conditions:  
\ben[(1)]
\item If $S$ is an $X$-set of  $G$ and  $S\subseteq S'$, then $S'$ is an $X$-set of $G$.
\item  The empty set is not an $X$-set of any graph.
\item An $X$-set of a disconnected graph is the union of an $X$-set of each component. 
\item If $G$ has no isolated vertices, then every set of $|V(G)|-1$ vertices is an $X$-set. 
\een  
\end{defn}

Note that $\Z(G), \Zp(G), \gamma(G),$ and $\pd(G)$ are all $X$-set parameters.
 Let  $X$ be an $X$-set parameter. The following definitions are adapted from \cite{PD-recon}.  The {\em $X$ token addition and removal  reconfiguration graph ($X$-TAR graph)}  of a  base graph  $G$, denoted by  $\xtar(G)$, has  the set of all  $X$-sets of $G$ as the set of vertices, and there is an edge between  two vertices  $S_1$ and $S_2$ of $\xtar(G)$ if and only if  $|S_1 \ominus S_2|=1$ where $A
\ominus B=(A\cup B)\setminus (A\cap B)$ denotes the  \textit{symmetric difference} of sets $A$ and $B$.  The {\em $k$-token addition and removal  reconfiguration graph for $X$} of $G$, denoted by  $\xtar_k(G)$, is the subgraph of $\xtar(G)$ induced by the set of  all  $X$-sets of cardinality at most $k$.
The maximum cardinality of a minimal $X$-set of  $G$ is the \emph{upper $X$-number} of $G$ and is denoted by $\xbar(G)$.

When the vertices of a TAR graph represent  a collection of solutions to a problem and an edge represents the application of a transformation rule from one solution to another, then that TAR graph is connected if and only  if   within the collection any solution can be transformed to any other  by repeated application of the transformation rule.   Observe that $\xtar(G)$ is connected for any graph $G$.  Thus determining  for which $k$ the  $k$-TAR is connected is a main theme of  the study of TAR graph reconfiguration  and connectedness is also a theme of other types of reconfiguration.
The least $k$ such that $\xtar_k(G)$ is connected is denoted by $\ulxo(G)$.  
The least $k$ such that $\xtar_i(G)$ is connected for $k\le i \le |V(G)|$ is denoted by $\xxo(G)$.

The disjoint union of graphs $G$ and $H$ is denoted by $G\sqcup H$ and $rK_1$ denotes $r$ isolated vertices.   As noted in \cite{PD-recon}, if   $G$ is a graph with has no isolated vertices, $G'=G\sqcup rK_1$, and $X$ is an $X$-set parameter, then   $X(G')=X(G)+r$, $\xtar_{k+r}(G')\cong\xtar_k(G)$, and $\xtar(G')\cong\xtar(G)$. 
Thus it suffices to study $X$-TAR reconfiguration   graphs with no isolated vertices. 

The \emph{order} of a graph $G$  is the cardinality of $V(G)$.  The \textit{maximum degree} and \textit{minimum degree} of  $G$ are denoted $\Delta(G)$ and $\delta(G)$ respectively.  It is well known (and easy to see) that $\delta(G)\le \Z(G)$. 
Let $G$ be a graph of order $n\ge 2$ with no isolated vertices.  As noted in \cite{PD-recon}, $\Delta(\xtar(G))=n$.  Observe that an $X$-set $S\ne V(G)$ of $G$ has $\deg_{\xtar_{|S|+1}(G)}(S)=n$ if and only if  $S\setminus\{x\}$ is an $X$-set for every $x\in S$. 

%==========================
\subsection{Zero forcing TAR graphs}\label{ss:Z-TAR}
Next we implement  definitions in Section \ref{ss:univ} for zero forcing.

\begin{defn}\label{d:ztar}
 The \emph{zero forcing token addition and removal ($\Z$-TAR) graph}     of a   base graph $G$, denoted by  $\ztar(G)$, has  the set of all   zero forcing sets of $G$ as the set of vertices, and there is an edge between  two vertices  $S_1$ and $S_2$ of $\ztar(G)$ if and only if  $|S_1 \ominus S_2|=1$.  The {\em  zero forcing $k$-token addition and removal  reconfiguration graph} of $G$, denoted by  $\ztar_k(G)$, is the subgraph of $\ztar(G)$ induced by the set of  all  zero forcing sets of cardinality at most $k$.
\end{defn}

\begin{defn}\label{d:zzo}
The maximum cardinality of a minimal zero forcing set of a graph $G$ is the \emph{upper zero forcing number} of $G$ and is denoted by $\zbar(G)$. The least $k$ such that $\ztar_k(G)$ is connected is denoted by $\ulzo(G)$.  The least $k$ such that $\ztar_i(G)$ is connected for each $i=k,\dots, |V(G)|$ 
is denoted by $\zzo(G)$.
\end{defn}

It was shown in \cite{smallparam} that 
 a graph with no isolated vertices has more than one minimum zero forcing set and no vertex is in every minimum zero forcing set. 
 Let  $G$ be a graph that has no isolated vertices.  Since a  minimal zero forcing set of cardinality $k$ has no neighbors in $\ztar_{k}(G)$,
$\ztar_{k}(G)$ is not connected whenever there is a minimal zero forcing set of cardinality $k$.  In particular, $\ztar_{\Z(G)}(G)$ and $\ztar_{\zbar(G)}(G)$ are not connected.  Brimkov and Carlson constructed an example of an infinite family of graphs $G_n$ with $\Z(G_n)=5$ and $\zbar(G_n)=n-2$ in \cite{BC22}.

Next we determine the zero forcing TAR reconfiguration graphs and related parameters for several families of graphs.  The complete graph is denoted by $K_n$ and $K_{1,r}$ denotes  the star with $r$ leaves (where a \emph{leaf} is a vertex of degree one); the \emph{center vertex} of $K_{1,r}$ is the vertex of degree $r$.

\begin{prop}\label{Kn}
For  $n\ge 2$, $\ztar(K_n) = K_{1,n}$.  Furthermore, $\zbar(K_n)=\Z(K_n)=n-1$ and $\zzo(K_n)=\ulzo(K_n)=n$.
\end{prop}
\begin{proof} It is well known (and immediate from $\delta(K_n)=n-1$) that $\Z(K_n)=n-1$ and the zero forcing sets of $K_n$ consist of all subsets of vertices of $K_n$ that contain at least $n-1$ vertices. Thus  $\ztar(K_n)=K_{1,n}$.   Since  $\ztar_{\Z(G)}(G)$ is disconnected for any  connected graph that has order at least two, $\zzo(K_n)=\ulzo(K_n)=n$.
\end{proof}

If $G_1$ and $G_2$ are disjoint, the \emph{Cartesian product} of $G_1$ and $G_2$, denoted by $G_1\square G_2$, is the graph with $V(G\square G')=V(G_1)\times V(G_2)$ such that $(v_1,v_2)$ and $(u_1,u_2)$ are adjacent if and only if 
 (1) $v_1=u_1$ and $v_2u_2\in E(G_2)$, or
(2)  $v_2=u_2$ and $v_1u_1\in E(G_1)$.
Figure \ref{ztarK13} illustrates the next result.

\begin{figure}[h!]
    \centering
    		\begin{tikzpicture}[scale=.7,every node/.style={draw,shape=circle,outer sep=2pt,inner sep=1pt,minimum
			size=.2cm}]
		
		\node[fill=black]  (00) at (1,1) {};
		\node[fill=black]  (01) at (-3,-0.5) {};
		\node[fill=black]  (02) at (-1,-0.5) {};
		\node[fill=black]  (03) at (3,-0.5) {};
		\node[fill=black]  (04) at (-3,1) {};
		\node[fill=black]  (05) at (-1,1) {};
		\node[fill=black]  (06) at (3,1) {};
		\node[fill=black]  (07) at (1,2.5) {};
		
		\node[draw=none] at (1, 1.5){$\small{\{1,2,3\}}$};
		\node[draw=none] at (-3,-1){$\small{\{1,2\}}$};
		\node[draw=none] at (-1,-1){$\small{\{2,3\}}$};
		\node[draw=none] at (3,-1){$\small{\{1,3\}}$};
		
		\node[draw=none] at (-3, 1.5){$\small{\{0,1,2\}}$};
		\node[draw=none] at (-1, 1.5){$\small{\{0,2,3\}}$};
		\node[draw=none] at (3, 1.5){$\small{\{0,1,3\}}$};
		\node[draw=none] at (1, 3){$\small{\{0,1,2,3\}}$};

		\draw[thick](00)--(03)--(06)--(07)--(04)--(01)--(00)--(02)--(05)--(07)--(00);
		\end{tikzpicture}
    \caption{$\ztar(K_{1,3})$}
    \label{ztarK13}
\end{figure}

\begin{prop}
Let $r\geq 2$. Then $\ztar(K_{1,r})=K_{1,r}\square K_2$.  Furthermore, $\zbar(K_{1,r})=\Z(K_{1,r})=r-1$ and  $\zzo(K_{1,r})=\ulzo(K_{1,r})=r$.
\end{prop}
\begin{proof}
Let  $V(K_{1,r}) = \{0,1,\ldots,r\}$ where $0$ is the unique vertex of degree $r$. The collection of zero forcing sets of $K_{1,r}$ can be partitioned as
\[
T_1 = \{V(K_{1,r})\} \cup \{S\subseteq V(K_{1,r}) : |S| = r \text{ and } 0\in S\}
\]
and 
\[
T_2 = \{V(K_{1,r})\setminus\{0\}\} \cup \{S\subseteq V(K_{1,r}) : |S| = r-1 \text{ and } 0\notin S\}.
\]
Observe that the induced subgraph of $\ztar(K_{1,r})$ on either $T_1$ or $T_2$ is $K_{1,r}$. Further, zero forcing sets $S_1\in T_1$ and $S_2\in T_2$ are adjacent in $\ztar(G)$ if and only if $S_2 = S_1\setminus\{0\}$. Thus  $\ztar(K_{1,r}) = K_{1,r} \square K_2 $.

It is immediate from the list of zero forcing sets that $\zbar(K_{1,r})=\Z(K_{1,r})=r-1$ and $\zzo(K_{1,r})=\ulzo(K_{1,r})=r$.
\end{proof}

\begin{rem} Let $n\ge 3$.  It is well known that a set   $S\subseteq V(C_n)$ is a zero forcing set of $C_n$ if and only if $S$ contains two consecutive vertices.  Thus $\Z(C_n)=\zbar(C_n)=2$ and $\zzo(C_n)=\ulzo(C_n)=3$.
\end{rem}

Next we apply  results established for TAR graphs of $X$-set parameters in \cite{PD-recon} and known properties of zero forcing to the zero forcing TAR graph and related parameters.  
 The $t$ dimensional hypercube $Q_t$ can be defined recursively in terms of the  Cartesian product as $
Q_0 = K_1 \mbox{ and }  Q_t = Q_{t-1}\square K_2.$

\begin{thm}{\rm \cite{PD-recon}}\label{t:univ-z1}
Let  $G$ be a graph of order $n\ge 2$ with no isolated vertices. Then 
\ben
\item
$\Delta(\ztar (G))=n$.
\item 
If $n\ge 2$, then $\delta(\ztar(G))=n-\zbar(G)$.
\item 
$\zbar(G)+1\le \zzo(G)\le\min\{\zbar(G)+\Z(G),n\}$.
\item  
If $\Z(G)=1$, then $\zzo(G)=\zbar(G)+1$
\item  
For every  $k\ge \Z(G)$, $\ztar_k (G)$ is a subgraph of $Q_n$.  Thus $\ztar_k (G)$  is bipartite.
\item No zero forcing TAR graph is isomorphic to a hypercube.
\een
\end{thm}

 The next result is immediate since $K_2=Q_1$.
\begin{cor}\label{no-ZTAR=K2}
There does not exist a graph $G$ such that  $\ztar(G) = K_2$.   \end{cor}

 Theorem \ref{t:univ-z1} allows us to show that the values of certain parameters of the base graphs of isomorphic zero forcing TAR graphs must be equal.
\begin{cor}\label{c:z-iso-props}{\rm \cite{PD-recon}}
Let $G$ and $G'$ be graphs with no isolated  vertices. If $\ztar(G) \cong \ztar(G')$, then
\ben
\item $|V(G)|=|V(G')|$,
\item $\Z(G)=\Z(G')$, and
\item $\zbar(G)=\zbar(G')$. 
\een 
\end{cor}

%==========================
% Section
%==========================

\section{Irrelevant sets and TAR graph isomorphisms}\label{s:xtar-iso}
The TAR reconfiguration graph of an $X$-set parameter carries a great deal of information about the $X$-sets of $G$, but the information about $X$-sets of $G$ does not generally determine $G$; in particular,  the information about   zero forcing sets of $G$ sets does not generally determine $G$.   

We show that for graphs $G$ and $G'$  with no isolated vertices, $\xtar(G)\cong\xtar(G')$ implies there is a relabeling of the vertices of $G'$ such that $G$ and $G'$ have exactly the same $X$-sets (see  Theorem \ref{isomorph-same-zfs}).  When applied to zero forcing, $\ztar(G)\cong\ztar(G')$ implies $G$ and $G'$ necessarily have the same the zero forcing polynomials (as defined in \cite{zfpolys}, see Section \ref{ss:zfpoly}). Showing that an isomorphism of zero forcing TAR graphs implies that their base graphs have the same the zero forcing polynomial triggered this study of isomorphisms.

Irrelevant sets are introduced in Section \ref{ss:iso-level} to prove that if $\xtar(G)\cong\xtar(G')$, then there is an isomorphism that preserves the cardinality of $X$-sets  (see  Theorem \ref{isomorph-isometry-u}).  The focus of Section \ref{ss:iso-level} is proving Theorem \ref{isomorph-same-zfs}, but irrelevant sets are interesting in their own right (and exhibit significant differences among various $X$-set parameters).  These ideas are discussed in Section \ref{ss:irrelevant}.  
 
In Section \ref{ss:unique} we establish  that certain zero forcing TAR graphs are unique up to isomorphism of the base graph (assuming the base graph has no isolated vertices), and also show that a cycle and a cycle with one chord have the same zero forcing TAR graphs (see Proposition \ref{p:Cn+e}).

\subsection{Isomorphisms of the TAR graph}\label{ss:iso-level}

 There is a natural poset structure on the
 vertices of $X$-TAR reconfiguration graphs whose partial ordering is determined by set inclusion. This poset structure can be used to better understand the isomorphisms between $X$-TAR reconfiguration graphs. We begin this subsection with a result showing that the vertices of induced hypercubes of $X$-TAR reconfiguration graphs are intervals in the aforementioned poset. To formally state this result, we recall some basic poset terminology found in \cite{stanley}.

 Let $Y$ be a set and $P=(Y,\leq)$ be a poset. An element $u\in Y$ is \textit{maximal} (respectively, \textit{minimal}) if $u\not< y$ (respectively, $u\not< y$) for each $y\in Y$. Elements $u,v\in Y$ are \textit{comparable} provided $u\leq v$ or $v\leq u$. If $u\leq v$ in $P$, then the set $[u,v] = \{y : u\leq y \leq v\}$ is called an \textit{interval}. A \textit{chain} is a poset such that each pair of elements is comparable. The \emph{length} of a chain $C$ is $|C|-1$ and the length of the interval $[u,v]$ is
 the maximum length of a chain in $[u,v]$ and is denoted by $\ell(u,v)$.

For a set $T$, let $\mathcal{P}(T)$ denote the power set of $T$. We remark that $\xtar(G)$ is a join-semilattice contained in $(\mathcal{P}(V(G)),\subseteq)$  and  $\ell(S,S')=\dist(S,S')$ for any interval $[S,S']$ in $\xtar(G)$ since $S$ and $S'$ are comparable (where $\dist(S,S')$ is the distance in the graph $\xtar(G)$).

\begin{lem}\label{interval}
Let $X$ be an $X$-set parameter, let $t\geq 0$ be an integer, let $G$ be a graph on $n$ vertices, and let $H$ be an induced subgraph of $\xtar(G)$.  If $H\cong Q_t$, then $V(H)$ is an interval  of length $t$ in the poset $(\mathcal{P}(V(G)),\subseteq)$.
\end{lem}
\bpf
The claim is obvious for $t=0,1$. Suppose that $t\geq 2$. We begin by showing that $(V(H),\subseteq)$ has exactly 1 maximal element. 

Assume, to the contrary, that $(V(H),\subseteq)$ has at least 2 maximal elements $u$ and $v$. Since $H$ is connected, there exists a path from $u$ to $v$ in $H$. Every path $P$ from $u$ to $v$ in $H$ can be written in the form $uy_1\cdots y_i v$. For each such $P$, let $d_P$ to be the smallest index $k$ such that $y_{k-1}\subseteq y_k$, where $y_0 = u$ and $y_{i+1} = v$. Note that since $u$ and $v$ are maximal in $(V(H),\subseteq)$, $d_P\geq 2$ is always defined.

Let $d$ be the minimum $d_P$ amongst all paths $P$ from $u$ to $v$ in $H$. Pick a path $P^*$ in $H$ of the form $ux_1\cdots x_i v$ such that $d_{P^*} = d$, and let $x_0 = u$ and $x_{i+1} = v$. Then $x_{d-1}\subseteq x_{d-2}$  because $x_{d-2}\not\subseteq x_{d-1}$ and $x_{d-2}$ is adjacent to $x_{d-1}$. Since $x_{d-2}$ and $x_d$ have  $x_{d-1}$ as a common neighbor and each pair of vertices in a hypercube share exactly 0 or 2 common neighbors, there exists some vertex $w \not=x_{d-1}$ in $H$ that is adjacent to $x_{d-2}$ and $x_d$. 
By our choice of $d$, $w\subseteq x_{d-2}$ and $w\subseteq x_d$   or else $P'=ux_1\dots x_{d-2}wx_{d}\dots x_iv$ would be a path with $d_{P'}<d_{P^*}$ since  $w=x_{d-2}\cup x_d $. 
Hence $x_{d-2}\cap x_d = w$. This is absurd since $x_{d-2}\cap x_d = x_{d-1}$ and $w\not=x_{d-1}$.

 Thus, $H$ has exactly 1 maximal element $T$. A similar argument shows $H$ has exactly 1 minimal element $S$ in the subset partial ordering of $\xtar(G)$. Thus, $S\subseteq R\subseteq T$ for every $R\in V(H)$.  Since $H\cong Q_t$, $\dist(S,T)\le \diam(Q_t)=t$.  So there are at most $2^t$ elements in the interval $[S,T]$.  But $|V(H)|=|V(Q_t)|=2^t$. Therefore, $\dist(S,T)=t$ and $V(H)=[S,T]$.
\epf

\begin{rem} 
We remark that the proof of Lemma \ref{interval} is really an argument about cover graphs. For a poset $(Y,\le)$ and  $x,y\in Y$, we say that $y$ covers $x$ provided  $x< y$ and there is no $z\in Y$ such that  $x< z < y$.  The \textit{cover graph} or \emph{Hasse diagram graph} of $P$ is the simple graph with vertices in $Y$ and an edge $uv$ if and only if $u$ covers $v$, or $v$ covers $u$. Lemma \ref{interval}  says that if $G\cong Q_t$, then $V(G)$ is an interval in $(\mathcal{P}(Y),\subseteq)$.\end{rem}

\begin{lem}\label{lem:ord_pres}
Let $X$ be an $X$-set parameter, let  $G$ and $G'$ be graphs. Suppose $\varphi:\xtar(G)\to \xtar(G')$ is an isomorphism. Let $S^\prime = \varphi(V(G))$. If $Y^\prime$ is a minimal $X$-set of $G^\prime$, then $Y^\prime\subseteq S^\prime$.
\end{lem}
\bpf
Let $Y^\prime$ be a minimal $X$-set of $G^\prime$. Define $Y = \varphi^{-1}(Y^\prime)$. The interval $[Y,V(G)]$ in $(\mathcal{P}(V(G),\subseteq)$ forms an induced $Q_t$ in $\xtar(G)$ for some integer $t>0$. By Lemma \ref{interval}, $\varphi([Y,V(G)])$ is an interval $[Z^\prime,W^\prime]$ in $(\mathcal{P}(V(G^\prime)),\subseteq)$. Since $Y^\prime$ is a minimal $X$-set of $G^\prime$, $Z^\prime = Y^\prime$. Thus, $Y^\prime \subseteq S^\prime$.
\epf

Some zero forcing TAR graphs have automorphisms that do not preserve the cardinality of the zero forcing sets that are the vertices of the TAR graph (see Example \ref{star-irrelevant-center}).  The next definition allows us to manage that issue.

\begin{defn}\label{d:irrelevant}
Let $G$ be a graph. 
We say that a vertex  $v\in  V(G)$ is \textit{$X$-irrelevant}  if $x\not\in S $ for every minimal $X$-set $S$ of $G$.  A set $R\subseteq V(G)$ is an \textit{$X$-irrelevant set}  if every vertex of $R$ is $X$-irrelevant.
\end{defn}

Irrelevant vertices exist for zero forcing as seen in the next example, but not for domination as discussed in Section \ref{ss:irrelevant}.
\begin{ex}\label{star-irrelevant-center}
Observe that the center vertex of $ K_{1,r}$ is a $\Z$-irrelevant vertex  and is in fact the only $\Z$-irrelevant vertex.
\end{ex}

\begin{defn}\label{d:nuR}Let $G$ be a graph and $R\subseteq V(G)$. We define the map 
\[
\nu_R: V(\xtar(G)) \to \mathcal{P}(V(G)) \quad \text{via} \quad \nu_R(S) = S\ominus R.
\]
\end{defn}

For a nonempty set $R$, the map $\nu_R$ maps $X$-sets to $X$-sets, and is thus  an automorphism of $\xtar(G)$. This is illustrated in the next example and shown in Theorem \ref{irrel_isom}. Observe that $\nu_R$ maps $S$ to $S'$ with $|S'|\ne |S|$.  

\begin{ex}\label{star-irrelevant-center2}
Consider $ K_{1,r}$ with center vertex $0$.   Recall that $\ztar( K_{1,r})= K_{1,r}\square K_2$ (see Figure \ref{ztarK13} for $\ztar(K_{1,3})$). For  $R=\{0\}$,    $\nu_R$ is the automorphism of $ K_{1,r}\square K_2$ obtained by reversing the two vertices of $K_2$. 
\end{ex}

\begin{thm}\label{irrel_isom}
Let $X$ be an $X$-set parameter, let $G$ be a graph, and let  $R\subseteq V(G)$. Then $\nu_R$ is a graph automorphism of $\xtar(G)$ if and only if $R$ is $X$-irrelevant.
\end{thm}
\bpf
Suppose that $\nu_R$ is an automorphism. By Lemma \ref{lem:ord_pres} every minimal $X$-set of $G$ is a subset of $\nu_R(V(G)) = V(G)\setminus R$. Thus, $R$ is $X$-irrelevant.

Suppose that $R$ is $X$-irrelevant. Let $S$ be an $X$-set of $G$. Then there exists some minimal $X$-set $T\subseteq S$. Since $R$ is $X$-irrelevant, $T\subseteq S\setminus R \subseteq S$. Thus, $\nu_R(S)\supseteq S\setminus R$ is an $X$-set of $G$. By construction, adjacency is preserved by $\nu_R$. Therefore, $\nu_R$ is an automorphism.
\epf

We restate Theorem \ref{irrel_isom} for the specific case of zero forcing.

\begin{cor}\label{irrel_isom_cor}
Let $G$ be a graph and let  $R\subseteq V(G)$. Then $\nu_R$ is a graph automorphism of $\ztar(G)$ if and only if $R$ is $\Z$-irrelevant.
\end{cor}

We are now ready to state one of our main results, namely that if there is an isomorphism between $X$-TAR graphs that does not preserve the sizes of the $X$-sets, then the difference in size is due to an  $X$-irrelevant set. Moreover, there is an isomorphism between these $X$-TAR graphs that preserves the size of the $X$-sets.

\begin{thm}\label{isomorph-isometry-u} Let $X$ be an $X$-set parameter, let  $G$ and $G'$ be graphs with no isolated vertices, and let  $\tilde\varphi:\xtar(G)\to\xtar(G^\prime)$ be an isomorphism.   Then $R^\prime = V(G^\prime)\setminus \tilde\varphi(V(G))$ is $X$-irrelevant and  $\varphi = \nu_{R^\prime}\circ \tilde\varphi$ is an isomorphism  such that $|\varphi(S)|=|S|$ for every $S\in V(\xtar(G))$.
\end{thm}
\bpf Let $n=|V(G)|$.
If $\tilde\varphi(V(G))=V(G')$, then $R'=\emptyset$ is $X$-irrelevant. So suppose that $\tilde\varphi(V(G)) = S^\prime$, where $S^\prime\not=V(G^\prime)$. Since $\tilde\varphi$ is an isomorphism, Lemma \ref{lem:ord_pres} implies every minimal $X$-set of $G^\prime$ is a subset of $S^\prime$ and hence $R^\prime$ is $X$-irrelevant. By Theorem \ref{irrel_isom}, $\nu_{R^\prime}$ is an automorphism of $\xtar(G^\prime)$. Thus, $\varphi = \nu_{R^\prime}\circ \tilde\varphi$ is an isomorphism such that $\varphi(V(G)) = V(G^\prime)$.

Note that $|V(G')|=n$ and $X(G')=X(G)$. Let $S\in V(\xtar(G))$.  The interval $H=[S,V(G)] \in \xtar(G)$ is an induced hypercube  in $\xtar(G)$, so $\varphi(H)$ is an induced hypercube in $\xtar(G')$.  By Lemma \ref{interval} and since $\varphi(V(G))=V(G')$, $\varphi(H)= [S',V(G')]$ in $\xtar(G')$ and $\dist(S',V(G'))=\dist(S,V(G))$. We show by induction on $|S|$ that $|\varphi(S)|=|S|$.
We say $S$ is a $k$-$X$-set if $S$ is an $X$-set and $|S|=k$.

For the base case, assume $|S|=X(G)$, so $\dist(S',T')=\dist(S,T)=n-X(G)=|V(G')|-X(G')$.  This implies $|S'|=X(G')=X(G)=|S|$. The same reasoning applies using $\varphi^{-1}$, since $\varphi^{-1}(V(G'))=V(G)$. 
Thus $\varphi$ defines a bijection between minimum  $X$-sets of $G$  and  minimum  $X$-sets  of $G'$.

Now assume $\varphi$ defines a bijection between  $i$-$X$-sets of $G$ and  $i$-$X$-sets of $G'$ for $X(G)\le i\le k$ and let $S$ be a $(k+1)$-$X$-set of $G$.  This implies  $|\varphi(W)|\ge k+1$ for $W\in [S,V(G)]$. By Lemma \ref{interval}, $\varphi([S,V(G)])= [S',V(G')]$ in $\xtar(G')$ and $\dist(S',V(G'))=n-k-1$.  Thus $S'$ is a $(k+1)$-$X$-set of $G$.
\epf

We restate Theorem \ref{isomorph-isometry-u} for the specific case of zero forcing.
\begin{cor}\label{cor:isomorph-isometry-u} Let   $G$ and $G'$ be graphs with no isolated vertices and let  $\tilde\varphi:\ztar(G)\to\ztar(G^\prime)$ be an isomorphism.   Then $R^\prime = V(G^\prime)\setminus \tilde\varphi(V(G))$ is $\Z$-irrelevant and  $\varphi = \nu_{R^\prime}\circ \tilde\varphi$ is an isomorphism  such that $|\varphi(S)|=|S|$ for every $S\in V(\ztar(G))$.
\end{cor}

In the next theorem, we show that if two graphs $G$ and $G'$  (with no isolated vertices)  have isomorphic $X$-TAR graphs, then there is a bijection between  the vertices of  $G$ and the vertices of $G'$   that results in the correspondence of the $X$-sets.   Note this bijection need not be a graph isomorphism between  $G$ and $G'$. A cycle and a cycle plus one edge provide an example of nonisomorphic graphs with the same zero forcing sets (see Proposition \ref{p:Cn+e}).  We first give some  useful notation.  For any map $\psi:A \to A'$ and subset $B\subseteq A$ we write $\psi(B)$ to mean the image of $B$, i.e.,\ $\psi(B) = \{\psi(b) : b\in B\}$. This is particularly useful when working with a map $\psi:V(G)\to V(G')$ that maps $X$-sets of $G$ to $X$-sets of $G'$, since this convention naturally induces a map $\psi:V(\xtar(G))\to V(\xtar(G'))$. For any graph $G$ and $v\in V(G)$, define $S_v=V(G)\setminus\{v\}$. 

\begin{thm}\label{isomorph-same-zfs} 
Let $X$ be an $X$-set parameter, let  $G$ and $G'$ be graphs with no isolated vertices, and suppose $\varphi:\xtar(G)\to \xtar(G')$ is a graph isomorphism. Then  $|\varphi(S)|=|S|$ for every $X$-set $S$ if and only if there exists  a bijection $\psi:V(G)\to V(G')$ such that $\psi(S)=\varphi(S)$ for every $X$-set $S$ of $G$.
\end{thm}
\bpf 
Let $n=|V(G)|$. For $v\in V(G)$ and  $v'\in V(G')$,  note that $S_v$ is an $X$-set of $G$ and $S_{v'}$ is an $X$-set of $G'$. 

Begin by assuming $|\varphi(S)|=|S|$ for every $X$-set $S$. Since $|\varphi(S_v)|=n-1$ for every $v\in V(G)$, we may define $\psi(v)=v'$ where $v'$ is the unique vertex such that $\varphi(S_v)=S_{v'}$. Note that $\psi:V(G)\to V(G')$ is a bijection.

 The proof that $\psi(S)=\varphi(S)$ for every $X$-set $S$ of $G$ proceeds iteratively from $|S|=n$ to $|S|=1$.  By the choice of $\varphi$ and the definition of $\psi$, we have $\varphi(S)=\psi(S)$ for $|S|=n,n-1$.  Assume $\varphi(S)=\psi(S)$ for each $X$-set $S$ of order $k$ for some $k$ with $n-1\ge k>X(G)$.  Let $S$ be an $X$-set of order $k-1$.   
 Since $k-1\le n-2$, there exist distinct vertices $a,b\in V(G)\setminus S$ such that $S\cup \{a\}$ and $S\cup \{b\}$ are $X$-sets of order $k$. Since $S$ is adjacent to $S\cup\{a\}$ and $S\cup\{b\}$ in $\xtar(G)$, and $|\varphi(S\cup\{a\})|=|\varphi(S\cup\{b\})| > |\varphi(S)|$, there exist distinct $a',b'\in V(G')\setminus \varphi(S)$ such that $\varphi(S\cup\{a\}) = \varphi(S)\cup\{a'\}$ and $\varphi(S\cup\{b\}) = \varphi(S)\cup\{b'\}$. Thus,
 \[
 \varphi(S)=\varphi(S\cup \{a\})\cap \varphi(S\cup \{b\})=\psi(S\cup \{a\})\cap \psi(S\cup \{b\})=\psi(S),
 \]
 where the last equality follows since $\psi$ is a bijection.
 
 Now assume there exists  a bijection $\psi:V(G)\to V(G')$ such that $\psi(S)=\varphi(S)$ for every $X$-set $S$ of $G$. Then $|S| = |\psi(S)| = |\varphi(S)|$ as desired.
 \epf
 
 We restate Theorem \ref{isomorph-same-zfs} for the specific case of zero forcing.
 \begin{cor}\label{cor:isomorph-same-zfs} 
 Let  $G$ and $G'$ be graphs with no isolated vertices, and suppose $\varphi:\ztar(G)\to \ztar(G')$ is a graph isomorphism. Then  $|\varphi(S)|=|S|$ for every zero forcing set $S$ if and only if there exists  a bijection $\psi:V(G)\to V(G')$ such that $\psi(S)=\varphi(S)$ for every  zero forcing set $S$ of $G$.
 \end{cor}

Theorem \ref {isomorph-same-zfs} also shows that if the domination (respectively,  power domination, PSD zero forcing) TAR reconfiguration graphs of $G$ and $G'$ are isomorphic, then $G'$ can be relabeled so that $G$ and $G'$ have exactly the same dominating (respectively,  power dominating, positive semidefinite  zero forcing) sets.

The next two results focus on mappings and minimal $X$-sets.

\begin{prop}\label{minimal-u} Let $X$ be an $X$-set parameter and let $G$ and $G'$ be graphs with no isolated vertices.   
 Suppose $\varphi:\xtar(G)\to \xtar(G')$ is a graph isomorphism  such that $|\varphi(S)|=|S|$. Then $\varphi$ maps minimal $X$-sets to minimal $X$-sets (of the same size).
\end{prop}
\bpf If $S$ is minimal $X$-set of $G$, then   $\deg_{\xtar(G)}(S)=n-|S|$.  If $T$ is an $X$-set of $G$ that is not minimal  in $\xtar(G)$, then   $\deg_{\xtar(G)}(T)\ge n-|T|+1$ (since $T$ has a subset neighbor).  Analogous statements are true for $G'$.
\epf

\begin{thm}\label{isomorph-minimal-u} Let $X$ be an $X$-set parameter and let $G$ and $G'$ be graphs with no isolated vertices.   Suppose $\psi:V(G)\to V(G')$ is a bijection.
\ben[$(1)$]
\item\label{c3}  Suppose $\psi$ maps  $X$-sets of $G$ to  $X$-sets of $G'$. Then the induced mapping $\psi:V(\xtar(G))\to V(\xtar(G'))$ is an isomorphism of $\xtar(G)$ and $\psi(\xtar(G))$.   If every  $X$-set of $G'$ is the image of an $X$-set of $G$, then $\psi:V(\xtar(G))\to V(\xtar(G'))$ is an isomorphism of $\xtar(G)$ and $\xtar(G')$.
\item\label{c2} Suppose $\psi$  maps minimal $X$-sets of $G$ to minimal $X$-sets of $G'$.  Then  $\psi$ maps $X$-sets of $G$ to $X$-sets of $G'$. 
If every minimal $X$-set of $G'$ is the image of a minimal $X$-set of $G$, then  $\psi$ is a bijection from $X$-sets of $G$ to $X$-sets of $G'$. 
\item\label{c4} Suppose $\psi$  maps minimal $X$-sets of $G$ to minimal $X$-sets of $G'$.   Then the induced mapping $\psi:V(\xtar(G))\to V(\xtar(G'))$ is an isomorphism of $\xtar(G)$ and $\psi(\xtar(G))$.   If every minimal  $X$-set of $G'$ is the image of a minimal $X$-set of $G$, then $\psi:V(\xtar(G))\to V(\xtar(G'))$ is an isomorphism of $\xtar(G)$ and $\xtar(G')$.\een
\end{thm}
\bpf 
 
\eqref{c3}: Since $\psi$ maps $X$-sets of $G$ to $X$-sets of $G'$,   $\psi$ induces a bijection between the vertices of $\xtar(G)$ and a subset of the vertices of  $\xtar(G')$ ($X$-sets of $G'$ of the form $\psi(S)$ where $S$ is an $X$-set of $G$). Assume that $S_1,S_2\in V(\xtar(G))$ are adjacent in $\xtar(G)$. Without loss of generality, $|S_1\setminus S_2|=1$. Since $\psi$ is a bijection, $|\psi(S_1)\setminus \psi(S_2)| = 1$. Thus, $\psi(S_1)$ and $\psi(S_2)$ are adjacent in $\xtar(G')$. Hence $\psi$ is an isomorphism from $\xtar(G)$ to $\psi(\xtar(G))$.

\eqref{c2}:
Let $S\in V(\xtar(G))$ be an $X$-set. There is a minimal $X$-set  $T\subseteq S$ of $G$ and $\psi(T)\subseteq \psi(S)$. Since $\psi(T)$ is a minimal $X$-set  of $G'$, $\psi(S)$ is an $X$-set   of $G'$.

Statement \eqref{c4} is immediate from statements \eqref{c2} and \eqref{c3}
\epf

It is possible to have a vertex bijection $\psi$ that maps minimal  zero forcing sets of $G$ to minimal zero forcing sets of $G'$ but not vice versa, as the next example shows (using the identity function as $\psi$).

\begin{ex} Consider the graph $K_{1,3}$ with vertices $\{0,1,2,3\}$ where 0 is the center vertex, and the paw graph $P$ constructed from $K_{1,3}$ by adding the edge $23$.  The minimal zero forcing sets of $K_{1,3}$ are $\{1,2\}$, $\{1,3\}$, and $\{2,3\}$. The minimal zero forcing sets of $P$ are $\{0,2\}$, $\{0,3\}$, $\{1,2\}$, $\{1,3\}$, and  $\{2,3\}$.
\end{ex}

	\begin{figure}[h!]
    \centering
    \scalebox{1}{
    		\begin{tikzpicture}[scale=.7,every node/.style={draw,shape=circle,outer sep=2pt,inner sep=1pt,minimum
			size=.2cm}]
		
		\node[fill=black]  (00) at (2,1) {};%123
		\node[fill=black]  (01) at (-2,-0.5) {};%12
		\node[fill=black]  (02) at (0, -0.5) {};%23
		\node[fill=black]  (03) at (4,-0.5) {};%13
		\node[fill=black]  (04) at (-2,1) {};%012
		\node[fill=black]  (05) at (0,1) {};%023
		\node[fill=black]  (06) at (4,1) {};%013
		\node[fill=black]  (07) at (2,2.5) {};%0123
		
		\node[draw=none] at (2, 1.5){$\small{\{1,2,3\}}$};
		\node[draw=none] at (-2,-1){$\small{\{1,2\}}$};
		\node[draw=none] at (0,-1){$\small{\{2,3\}}$};
		\node[draw=none] at (4,-1){$\small{\{1,3\}}$};

		\node[draw=none] at (-2, 1.5){$\small{\{0,1,2\}}$};
		\node[draw=none] at (0, 1.5){$\small{\{0,2,3\}}$};
		\node[draw=none] at (4, 1.5){$\small{\{0,1,3\}}$};
		\node[draw=none] at (2, 2.8){$\small{\{0,1,2,3\}}$};

		\draw[thick](00)--(03)--(06)--(07)--(04)--(01)--(00)--(02)--(05)--(07)--(00);
		\end{tikzpicture}}
		\scalebox{1}{	\begin{tikzpicture}[scale=.7,every node/.style={draw,shape=circle,outer sep=2pt,inner sep=1pt,minimum
			size=.2cm}]
		
		\node[fill=black]  (00) at (2,1) {};%123
		\node[fill=black]  (01) at (-2.8,-0.5) {};%12
		\node[fill=black]  (02) at (0, -0.5) {};%23
		\node[fill=black]  (03) at (4,-0.5) {};%13
		\node[fill=black]  (04) at (-2.8,1) {};%012
		\node[fill=black]  (05) at (0,1) {};%023
		\node[fill=black]  (06) at (4,1) {};%013
		\node[fill=black]  (07) at (2,2.5) {};%0123
		
		\node[fill=black]  (08) at (-1.4,-0.5) {};%02
		\node[fill=black]  (09) at (2,-0.5) {};%03
		
		\node[draw=none] at (2, 1.5){$\small{\{1,2,3\}}$};
		\node[draw=none] at (-2.8,-1){$\small{\{1,2\}}$};
		\node[draw=none] at (0,-1){$\small{\{2,3\}}$};
		\node[draw=none] at (4,-1){$\small{\{1,3\}}$};
		\node[draw=none] at (-1.4,-1){$\small{\{0,2\}}$};
		\node[draw=none] at (2,-1){$\small{\{0,3\}}$};
		
		\node[draw=none] at (-2.8, 1.5){$\small{\{0,1,2\}}$};
		\node[draw=none] at (0, 1.5){$\small{\{0,2,3\}}$};
		\node[draw=none] at (4, 1.5){$\small{\{0,1,3\}}$};
		\node[draw=none] at (2, 2.8){$\small{\{0,1,2,3\}}$};

		\draw[thick](00)--(03)--(06)--(07)--(04)--(01)--(00)--(02)--(05)--(07)--(00);
		\draw[dashed] (06)--(09)--(05)--(08)--(04);
		\end{tikzpicture}}
    \caption{$\ztar(K_{1,3})$ and $\ztar(P)$ }
    \label{ztarK13ztarpaw}
\end{figure}

%==========================
% Subsection
%==========================

\subsection{$X$-irrelevant vertices and automorphisms of $X$-TAR graphs.}\label{ss:irrelevant}

In this section we explore the existence (or non-existence) of  $X$-irrelevant vertices for various $X$-set parameters and consequences for automorphisms of $X$-TAR graphs.
The following proposition is an application of  Theorems \ref{irrel_isom} and \ref{isomorph-isometry-u}.

\begin{prop}
Let $G$ be a graph with no isolated vertices. If there exists an automorphism $\varphi\in\text{aut}(G)$ that does not fix $V(G)$, then $\xtar(G)$ has a perfect matching.
\end{prop}
\bpf 
Suppose that there exists an automorphism $\varphi\in\text{aut}(G)$ that does not fix $V(G)$. By Theorem \ref{isomorph-isometry-u}, $V(G)\setminus \varphi(V(G))$ is $X$-irrelevant. Let $R$ be any single element subset of $V(G)\setminus \varphi(V(G))$. Then $R$ is $X$-irrelevant and by Theorem \ref{irrel_isom} $\nu_{R}$ is an automorphism of $\xtar(G)$. Observe that no vertex is fixed by $\nu_{R}$ and $\nu_{R}^2$ is the identity map. Thus, $\nu_{R}$ induces a perfect matching on $\xtar(G)$.
\epf

 Let $M_X(G)$ denote the set of bijections 
 $\psi:V(G)\to V(G)$ that send minimal $X$-sets of $G$ to minimal $X$-sets of $G$ of the same size.

\begin{thm}
Let $G$ be a graph with no isolated vertices. Then the automorphism group of $\xtar(G)$ is generated by
\[
\{\nu_R : R\text{ is $X$-irrelevant}\}\cup M_X(G).
\]
\end{thm}
\bpf
By Theorem \ref{irrel_isom} $\nu_R\in\text{aut}(\xtar(G))$ for every $X$-irrelevant set $R$. 
By Theorem \ref{isomorph-minimal-u}\eqref{c4}, $\psi\in\text{aut}(\xtar(G))$ for every  $\psi\in M_X(G)$ (since $G'=G$ here, having $\psi$ map minimal $X$-sets to minimal $X$-sets is sufficient).

We now show that $\{\nu_R : R\text{ is $X$-irrelevant}\}\cup M_X(G)$ generates aut$(\xtar(G))$. Let $\varphi$ by an automorphism of $\xtar(G)$. Suppose that $V(G)$ is fixed by $\xtar(G)$. Then $|\varphi(S)| = |S|$ for each $S\in V(\xtar(G))$. By Theorem \ref{isomorph-same-zfs} there exists a bijection $\psi:V(G)\to V(G)$ such that $\psi(S) = \varphi(S)$ for every $X$-set $S$ of $G$. By Proposition \ref{isomorph-minimal-u}, $\psi$ maps minimal $X$-sets to minimal $X$-sets of the same size.

Suppose that $V(G)$ is not fixed by $\varphi$. By Theorem \ref{isomorph-isometry-u} and the preceding argument, there exists a bijection $\psi\in M_X(G)$ such that $\psi = \nu_{R}\circ\varphi$, where $R=V(G)\setminus \varphi(G)$. Thus, $\varphi = \nu_R^{-1}\circ\psi$.
\epf

\begin{ex}
Consider $P_4$ with vertices $1,2,3,4$ and edges 
$12,23,34$. Define $\psi:V(G)\to V(G)$ by $\psi(1)=4$, $\psi(2)=2$, $\psi(3)=3$ and $\psi(1)= 4$.  The minimal zero forcing sets are $\{1\}$, $\{4\}$, and $\{2,3\}$, so $\psi$ maps minimal forcing sets to minimal forcing sets.  Thus $\psi$ defines  an automorphism of $\ztar(P_4)$.  However, $\psi$  is \emph{not} an automorphism of $P_4$. 
\end{ex}

The next result provides many examples of graphs with nonempty $\Z$-irrelevant sets (and includes $K_{1,r}$ discussed in Example \ref{star-irrelevant-center}).   Let $H$ and $G$ be graphs, and for each $v\in V(H)$, let $G_v$ denote a copy of $G$ such that $H$ and $G_v, v\in V(H)$ are all disjoint graphs. The \emph{corona} of $H$ with $G$, denoted by $H\circ G$, has $V(H\circ G)=V(H)\cup\bigcup_{v\in V(H)}V(G_v)$ and $E(H\circ G)=E(H)\cup\bigcup_{v\in V(H)}E(G_v)\cup\bigcup_{v\in V(H)}\{vx:x\in V(G_v)\}$.  The graph $K_3\circ 2K_1$ is shown in Figure \ref{fig:K3circ2K1}.

\begin{prop}
Let $H$ be a connected graph and $G=H\circ rK_1$ with $r\ge 2$.  Then $T\subseteq V(G)$ is a $\Z$-irrelevant set of $G$ if and only if $T\subseteq V(H)$.
\end{prop}
\bpf Let $V(H)=\{v_1,\dots,v_k\}$ and denote the leaves adjacent to $v_i$ by $x_{i,j}$ for $j=1,\dots,r$.  Then $S\subseteq V(G)$ is a  zero forcing set of $G$ if and only if $S$ contains at least $r-1$ of the vertices $x_{i,1},\dots,x_{i,r}$ for all  $i=1,\dots,k$.  Thus $T\subseteq V(G)$ is $\Z$-irrelevant set if and only if $T\subseteq V(H)$.
\epf

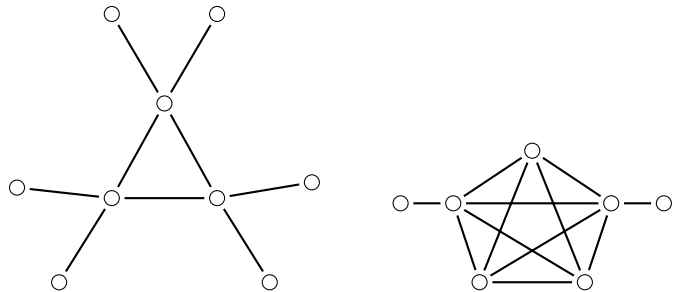
\begin{figure}[!h]
    \centering
    \scalebox{1}{
    \begin{tikzpicture}[scale=.7,every node/.style={draw,shape=circle,outer sep=2pt,inner sep=1pt,minimum
			size=.2cm}]
		
		\node[fill=none]  (00) at (-2.8,0.2) {};
		\node[fill=none]  (01) at (-1,0) {};
		\node[fill=none]  (02) at (1,0) {};
		\node[fill=none]  (03) at (2.8,0.3) {};
		\node[fill=none]  (04) at (0,1.8) {};
		\node[fill=none]  (05) at (-1,3.5) {};
		\node[fill=none]  (06) at (1,3.5) {};
		\node[fill=none]  (07) at (-2,-1.6) {};
		\node[fill=none]  (08) at (2,-1.6) {};
		
		%Drawing the thick edges
		\draw[thick] (05)--(04)--(02)--(08);
		\draw[thick] (06)--(04)--(01)--(07);
		\draw[thick](00)--(01)--(02)--(03);
		\end{tikzpicture}}
		\qquad
		\scalebox{1}{\begin{tikzpicture}[scale=.7,every node/.style={draw,shape=circle,outer sep=2pt,inner sep=1pt,minimum
			size=.2cm}]
		
		\node[fill=none]  (0) at (-1.5,2) {};
		\node[fill=none]  (1) at (0,3) {};
		\node[fill=none]  (2) at (1.5,2) {};
		\node[fill=none]  (3) at (1,0.5) {};
		\node[fill=none]  (4) at (-1,0.5) {};
		\node[fill=none]  (5) at (-2.5,2) {};
        \node[fill=none]  (6) at (2.5,2) {};
		
		%Drawing the thick edges
		\draw[thick] (0)--(1)--(2)--(3)--(4)--(0)--(2)--(4)--(1)--(3)--(0);
		\draw[thick] (0)--(5);
	    \draw[thick] (2)--(6);
		\end{tikzpicture}}
    \caption{The graphs $K_3\circ 2K_1$ and $G(5,2)$ }
    \label{fig:K3circ2K1}
\end{figure}

 Next we apply the results in Section \ref{ss:iso-level} to other $X$-set parameters.
The power domination TAR reconfiguration graph of a graph $G$ was introduced  in \cite{PD-recon}. We define a family  of graphs with nonempty $\pd$-irrelevant sets. For $r\ge 3$ and $1\le \ell \le r-2$, construct $G(r,\ell)$  from $K_r\circ K_1$ by deleting $r-\ell$ leaves.  

\begin{prop}
For $r\ge 3$ and $1\le \ell\le r-2$, let $L$ denote the set of leaves of $G(r,\ell)$.  Then $T\subseteq V(G(r,\ell))$ is a $\pd$-irrelevant set of $G(r,\ell)$ if and only if $T\subseteq L$. 
\end{prop}
\bpf
The set $L$ is not a power dominating set of $G(r,\ell)$ because no vertex dominated by $L$ is adjacent to exactly one vertex of $G(r,\ell)$ that is not dominated by $L$, and $L$ does not dominate $G$.  But any one vertex of $K_r$ is a power dominating set.
\epf

The domination TAR reconfiguration graph of a graph $G$, denoted by $\dtar(G)$, was studied in \cite{dom-recon} after $\dtar_k(G)$ was introduced in \cite{HS14}. 
The next remark is known.

\begin{rem}\label{r:all-min-dom}
Let $G$ be a graph and $v\in V(G)$.  A minimal dominating set containing $v$ can be  constructed by starting with $S=\{v\}$ and repeatedly adding $w\not\in N[S]$ until $N[S]=V(G)$.
\end{rem}

Remark  \ref{r:all-min-dom} immediately implies the next result.

\begin{prop} \label{p:no-dom-irrel}
  If $G$ is a graph and $S\subseteq V(G)$ is $\gamma$-irrelevant, then $S=\emptyset$.
\end{prop}

The next result is immediate from Proposition \ref{p:no-dom-irrel} and Theorems \ref{isomorph-isometry-u} and \ref{isomorph-same-zfs}.

\begin{cor}  Let  $G$ and $G'$ be graphs with no isolated vertices, and suppose $\varphi:\dtar(G)\to \dtar(G')$ is a graph isomorphism. For every  dominating set $S$ of $G$, $|\varphi(S)|=|S|$.  Furthermore,  there exists  a bijection $\psi:V(G)\to V(G')$ such that $\psi(S)=\varphi(S)$ for every  dominating set $S$ of $G$.
\end{cor}

Let $\psdztar(G)$ denote the positive semidefinite zero forcing TAR reconfiguration graph of $G$, which has  positive semidefinite zero forcing sets as vertices.  No work on this reconfiguration graph has appeared, but  the positive semidefinite zero forcing number is an $X$-set parameter.  Known results about positive semidefinite zero forcing and Theorems \ref{isomorph-isometry-u} and \ref{isomorph-same-zfs} provide information about isomorphisms of $\psdztar(G)$.

\begin{prop}{\rm \cite{EGR11}}
Let $G$ be a graph. Then for any vertex $v\in V(G)$, there exists a minimum positive semidefinite zero forcing set $S$ such that $v\in S$.
\end{prop}

\begin{cor} \label{p:no-Zp-irrel}
Let  $G$ and $G'$ be graphs with no isolated vertices, and suppose $\varphi:\psdztar(G)\to \psdztar(G')$ is a graph isomorphism. For every positive semidefinite zero forcing set $S$ of $G$, $|\varphi(S)|=|S|$.  Furthermore, there exists  a bijection $\psi:V(G)\to V(G')$ such that $\psi(S)=\varphi(S)$ for every positive semidefinite zero forcing set $S$ of $G$.
\end{cor}

%==========================
% Subsection
%==========================

\subsection{Zero forcing TAR graphs and zero forcing polynomials}\label{ss:zfpoly}
Let $G$ be a graph of order $n$ and let $z(G;k)$ denote the number of zero forcing sets of cardinality $k$.  Boyer et al.~defined  the \emph{zero forcing polynomial} of $G$ to be $Z(G;x)=\sum_{k=\Z(G)}^n z(G;k)x^k$ in \cite{zfpolys}.  The next result is immediate from Corollary \ref{cor:isomorph-isometry-u} (or Corollary \ref{cor:isomorph-same-zfs}).

\begin{cor}\label{isomorph-zfpoly} Let $G$ and $G'$ be graphs with no isolated vertices such that $\ztar(G)\cong \ztar(G')$.  Then the zero forcing polynomials of $G$ and $G'$ are equal. \end{cor}

The converse to Corollary \ref{isomorph-zfpoly} is false.     
Small examples of graphs with the same zero forcing polynomial and nonisomorphic zero forcing TAR graphs are easy to find with software, and one such example is presented next.

\begin{figure}[h!]
    \centering 
    \begin{tikzpicture}[scale=.65,every node/.style={draw,shape=circle,outer sep=2pt,inner sep=1pt,minimum size=.2cm}]   
	
	\node[fill=none]  (5) at (1,1) {\small{5}};
	\node[fill=none]  (2) at (1,-1) {\small{2}};
	\node[fill=none]  (6) at (-3,-1) {\small{6}};
	\node[fill=none]  (4) at (-3,1) {\small{4}};
	\node[fill=none]  (1) at (0,0) {\small{1}};
	\node[fill=none]  (3) at (-2,0) {\small{3}};
	\node[draw=none] at (-1, -2){\small{$G$}};
	\draw[thick] (4)--(3)--(6);
	\draw[thick] (3)--(1)--(5)--(2)--(1);
	
	\node[fill=none]  (2) at (4,1) {\small{2}};
	\node[fill=none]  (3) at (4,-1) {\small{3}};
	\node[fill=none]  (4) at (6,-1) {\small{4}};
	\node[fill=none]  (5) at (6,1) {\small{5}};
	\node[fill=none]  (1) at (3,0) {\small{1}};
	\node[fill=none]  (6) at (7,0) {\small{6}};
	\node[draw=none] at (5, -2){\small{$H$}};
	\draw[thick] (4)--(2)--(3)--(5)--(2)--(1)--(3)--(4)--(5)--(6)--(4);
	
	\end{tikzpicture}
            \caption{Two graphs the the same zero forcing polynomial and nonisomorphic zero forcing TAR graphs }  \label{fig:TAR-zfpoly}
\end{figure}
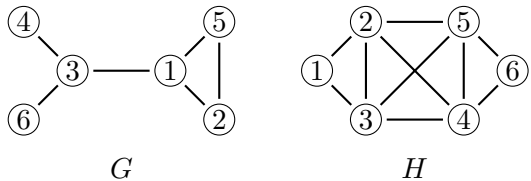

\begin{ex}
The graphs $G$ and $H$ in Figure \ref{fig:TAR-zfpoly} have \[Z(G;x)=Z(H;x)=x^6+6x^5+13x^4+8x^3\] (see \cite{sage:z-recon} for the computations).  The minimal zero forcing sets of $G$ are \{1, 2, 4\},
 \{1, 2, 6\},
 \{1, 4, 5\},
 \{1, 5, 6\},
 \{2, 4, 5\},
 \{2, 4, 6\},
 \{2, 5, 6\},
 \{4, 5, 6\} and the minimal zero forcing sets of $H$ are  \{1, 2, 4\},
 \{1, 2, 5\},
 \{1, 3, 4\},
 \{1, 3, 5\},
 \{2, 4, 6\},
 \{2, 5, 6\},
 \{3, 4, 6\},
 \{3, 5, 6\},
 \{2, 3, 4, 5\}.  Thus $\zbar(G)=3$, $\zbar(H)=4$, and $\ztar(G)\not\cong\ztar(H)$.  
 \end{ex}

 \begin{cor}
Let $G$ and $G'$ be graphs with no isolated vertices. Suppose that $\psi:V(G)\to V(G')$ is a bijection that maps minimal zero forcing sets of $G$ to minimal zero forcing sets of $G'$.  Then the coefficients of $Z(G,x)$ are bounded above by the coefficients of $Z(G',x)$.
\end{cor}
\bpf  By Theorem \ref{isomorph-minimal-u}\eqref{c2}, $\psi$ maps zero forcing sets of $G$ to zero forcing sets of $G'$.  Since $\psi$ is a bijection, $S$ and $\psi(S)$ have the same cardinality.  Thus  $G'$ must have at least as many  zero forcing sets of  cardinality $k$ as $G$ does.
\epf

%==========================
% Subsection
%==========================

\subsection{Uniqueness and nonuniqueness}\label{ss:unique}
In this section we present examples of nonisomorphic graphs $G$ and $H$ with no isolated vertices such that $\ztar(G)\cong\ztar(H)$. Recall that by Theorem \ref{isomorph-same-zfs}, this implies the vertices of $G$ can be labelled  so that $G$ and $H$ have the same zero forcing sets. We also present examples of graphs $H$ (with no isolated vertices) for which the zero forcing TAR reconfiguration is unique, meaning that for any graph $G$ with no isolated vertices, $\ztar(G)\cong\ztar(H)$  implies $G\cong H$.
For a graph $G$ and vertices $u$ and $v$ that are not adjacent in $G$, the graph $G+uv$ is the graph with vertex set $V(G)$ and edge set $E(G)\cup\{uv\}$.

\begin{prop}\label{p:Cn+e}
Let  $n\ge 4$. Then $\ztar(C_n) =\ztar(C_n+uv)$ for any two vertices $u$ and $v$  of $C_n$ that are not adjacent in $C_n$.
\end{prop}
\bpf
For both $C_n$ and $C_n+uv$, a set of vertices $S$ is a zero forcing set if and only if $S$ contains two vertices that are consecutive on the cycle.  Since $C_n$ and $C_n+uv$ have exactly the same zero forcing sets, $\ztar(C_n) =\ztar(C_n+uv)$.
\epf

Next we present examples of unique zero forcing TAR graphs.  The first two are immediate since they are the unique graphs having the given zero forcing number.  However, the third is less obvious.

\begin{prop}\label{c:path-unique} Let $G$ be a graph on $n\ge 2$ vertices with no isolated vertices.  If $\ztar(G)\cong\ztar(P_n)$, then $G\cong P_n$. \end{prop} 
\bpf
By Corollary \ref{c:z-iso-props}, $\ztar(G)\cong\ztar(P_n)$ implies $\Z(G)=\Z(P_n)=1$, which implies $G$ is a path.
\epf

\begin{prop}
Let $G$ be a connected graph on $n\ge 2$ vertices with no isolated vertices. Then $\ztar(G) \cong K_{1,n}$ implies $G \cong K_{n}$.
\end{prop}
\bpf By Corollary \ref{c:z-iso-props}, $\ztar(G)\cong K_{1,n}\cong\ztar(K_n)$ implies $\Z(G)=\Z(K_n)=n-1$, which implies $G$ is a complete graph.
\epf

In order to discuss uniqueness of the next family of graphs, we need to know what happens in the case of disconnected graphs. 

\begin{prop}\label{disjoin - cartesian-u}
Let $X$ be an $X$-set parameter and let $G = G_1\sqcup G_2$. Then $\xtar(G) = \xtar(G_1)\square\xtar(G_2)$.
\end{prop} 
\begin{proof}
Let $S$ and $T$ be $X$-sets of $G$. Then $S = S_1\sqcup S_2$ and $T = T_1 \sqcup T_2$, where $S_1$ and $T_1$ are zero forcing sets of $G_1$, and $S_2$ and $T_2$ are zero forcing sets of $G_2$. Observe that $S$ and $T$ are adjacent in $\ztar(G)$ if and only if there exists a vertex $v\in V(G_1)\sqcup V(G_2)$ such that
\begin{itemize}
    \item $S_1 = T_1$ and ($ T_2=S_2\setminus \{v\}$ or $S_2 = T_2\setminus \{v\}$), or
    \item $S_2 = T_2$ and ($ T_1=S_1\setminus \{v\} $ or $S_1 = T_1\setminus \{v\}$).
\end{itemize}
Further, $ T_2=S_2\setminus \{v\}$ or $S_2 = T_2\setminus \{v\}$ if and only if $S_2$ is adjacent to $T_2$ in $\xtar(G_2)$, and $ T_1=S_1\setminus \{v\} $ or $S_1 = T_1\setminus \{v\}$ if and only $S_1$ is adjacent to $T_1$ in $\xtar(G_1)$.
\end{proof}

We restate Proposition \ref{disjoin - cartesian-u} for the specific case of zero forcing. 

\begin{cor}\label{disjoin - cartesian}
Let $G = G_1\sqcup G_2$. Then $\ztar(G) = \ztar(G_1)\square\ztar(G_2)$.
\end{cor}

The proof of the next result uses some additional definitions and known properties of zero forcing. Let $G$ be a graph. For a given zero forcing set $S$, carry out a forcing process to color all vertices blue, recording the forces; the {set  of these forces} is denoted by $\F$. A set of forces $\F$ of $S$ defines a \emph{reversal}  of $S$, namely  the set of vertices that do not perform a force (using the set of forces $\F$). The next process is sometimes called \emph{neighbor trading} \cite{GHH}. Let $S$ be a minimum zero forcing set of $G$ and $v\in S$ with $\deg_G(v)\ge 2$. Suppose   $v\to w$ can be the first force performed in a forcing process.  Let $u\in N_G(v)$ and $u\ne w$.  Then $u\in S$ and $S\setminus \{u\}\cup \{w\}$ is a  minimum zero forcing set with first force $v\to w$ by $v\to u$ and the other forces remaining the same.  

\begin{prop}\label{p:star-unique} Let $G$ be a graph on $n\ge 3$ vertices with no isolated vertices.  If $\ztar(G)\cong K_{1,r}\square K_2$, then $G\cong K_{1,r}$. \end{prop} 
\bpf Suppose $\ztar(G)\cong K_{1,r}\square K_2$. Recall that $\ztar(H_1\sqcup H_2)\cong \ztar(H_1)\square\ztar(H_2)$.  Since there does not exist a graph $H_2$ with $\ztar(H_2)=K_2$, we may assume $G$ is connected. 

Since $\ztar(G)\cong \ztar(K_{1,r})$, the order of $G$ is $r+1$ and we label the vertices of $G$ so that the zero forcing sets of $G$ are exactly the zero forcing sets of $K_{1,r}$ where the vertices of both graphs are $\{0,\dots,r\}$ and $0$ is the center vertex  of $K_{1,r}$. Note that every minimal zero forcing set of $K_{1,r}$, and thus of $G$, is a minimum zero forcing set. Recall from  Example \ref{star-irrelevant-center} that the center vertex $0$ is $\Z$-irrelevant, i.e., not in any minimum zero forcing set.  

  Let $B$ be a minimum zero forcing set of $G$.  Since $\Z(G)=|V(G)|-2$,  exactly two forces are performed to color all vertices blue starting with the vertices in $B$ blue.
Since a vertex that does not force is in the zero forcing set of the reversal of the set of forces, $0$ must perform a force and the set of forces must be $\{i\to 0, 0\to j\}$.  If $\deg_G i\ge 2$, then by neighbor trading $0$ would be in a minimum zero forcing set.  Thus $\deg_G i=1$.    By considering the reversal, $\deg_G j=1$. If $n=3$, then $G\cong K_{1,2}$, so assume $n\ge 4$ (i.e., $r\ge 3$).  Then there is another vertex $k$, which must be adjacent to $0$ and not adjacent to $i$ or $j$.  If $r=3$, then $G\cong K_{1,r}$.  If $r\ge 4$ then there exists another vertex $\ell$.  Since $\ell\not\in N_G(i)$ and $\ell\not\in N_G(j)$, $\ell\in N_G(0)$ or $\ell\in N_G(k)$. Let $\ltimes$ denote the graph obtained from $K_3$ by adding adding two leaves to one vertex. If $\ell\in N_G(0)\cap N_G(k)$, then $G[\{0,i,j,k,\ell\}]\cong\ltimes$, which is a forbidden induced subgraph for $\Z(G)=|V(G)|-2$.  If $\ell\in N_G(k)$ and $\ell\not\in N_G(0)$, then $G[\{0,i,k,\ell\}]\cong P_4$,  which is a forbidden induced subgraph for $\Z(G)=|V(G)|-2$. Thus $G[\{0,i,j,k,\ell\}]\cong K_{1,4}$.  The  argument for $\ell$ shows any additional vertices must also be leaves and $G\cong K_{1,r}$. 
\epf

 The next table shows the number of (nonismorphic) graphs without isolated vertices of order at most eight  that have  unique Z-TAR reconfiguration graphs.
 
\begin{table}[!h]
\begin{center}
\begin{tabular}{ |c|c|c|c|c|c|c|c|c| } 
 \hline
 vertices  & 2 & 3 & 4 & 5 & 6 & 7 
 & 8\\
\hline
\# graphs with no isolated vertices  & 1 & 2 & 7 & 23 & 122 & 888 & 11302\\
 \hline
\# graphs with  unique $\ztar(G)$  & 1 & 2 & 4 & 7 & 34 & 303 & 5318\\
  \hline
 ratio (\# unique/\# no isolated) &  1 & 1 & 0.5714 & 0.3043 & 0.2787 & 0.3412 & 0.4705\\
 \hline
\end{tabular}\vspace{-9pt}
\end{center}
    \caption{Number of graphs with unique zero forcing TAR graph for small orders}
    \label{t:unique-tar-data}
\end{table}

%==========================
%
% Section
%
%==========================
\section{Connectedness properties of the zero forcing TAR graph}\label{s:ZTAR}

The focus of this section is connectedness properties of the zero forcing TAR graph.  We exhibit a family of graphs $H$ where $\zzo(H)$ exceeds the lower bound $\zbar(H)+1$  by an arbitrary amount and another family of graphs $G$ where $\ulzo(G)$ is strictly less than $\zzo(G)$. These examples are interesting because the more common situation is $\ulzo(G)=\zzo(G)=\zbar(G)+1$. Since the next result applies to several graph parameters,  we state it in the universal format (the result is known for domination  \cite{HS14}).

\begin{prop}\label{p:upper-con}
Let $X$ be an $X$-set parameter and let $G$ be a graph. If $\xtar_{k}(G)$ is connected for some $k>\xbar(G)$, then $\xtar_{k'}(G)$ is connected  for $k'=k,\dots,n$ and $\xxo(G)\le k$.   If $\xtar_{\xbar(G) + 1}(G)$ is connected, then $\xxo(G) = \xbar(G) + 1$.  If $\ulxo(G)>\xbar(G)$, then $\ulxo(G)=\xxo(G)$.
\end{prop}
\begin{proof}
 Suppose $\xtar_{k}(G)$ is connected for some $k \geq \xbar(G) + 1$. Consider two distinct $X$-sets $S_1', S_2' \in V(\xtar_{k'}(G))$ with $k'\ge k$. Let  $S_1$ and $S_2$ be minimal $X$-sets of $G$ contained in $S_1'$ and $S_2'$, respectively, and note that $S_1, S_2 \in V(\xtar_{k}(G))$. Since $\xtar_{k}(G)$ is connected, there exists a path in $\xtar_{k}(G)$ from $S_1$ to $S_2$.  Furthermore,  there is a path $P_i$ from $S_i$ to $S_i'$ in $\xtar_{k'}(G)$ for  $i=1,2$. Since $\xtar_{k}(G)$ is a subgraph of $\xtar_{k'}(G)$, there exists a path from $S_1'$ to $S_2'$ in $\xtar_{k'}(G)$. Thus, $\xtar_{k'}(G)$ is connected.  This implies $\xxo(G)\le k$.  When $k=\xbar(G)+1$, the  bound $\xxo(G)\le\xbar(G)+1$ \cite{PD-recon} then implies $\xxo(G) = \xbar(G) + 1$.  The last statement is immediate from the first.
\end{proof}

\begin{cor}\label{c:upper-con-z}
Let $G$ be a graph. If $\ztar_{k}(G)$ is connected for some $k>\zbar(G)$, then $\ztar_{k'}(G)$ is connected  for $k'=k,\dots,n$ and $\zzo(G)\le k$.   
 If $\ztar_{\zbar(G) + 1}(G)$ is connected, then $\zzo(G) = \zbar(G) + 1$.
 If $\ulzo(G)>\zbar(G)$, then $\ulzo(G)=\zzo(G)$.
\end{cor}

\begin{ex} Let $n\geq 4$ and consider the path $P_n$ with vertices in path order. Then $S\subseteq V(P_n)$ is a zero forcing set if and only if $S$ contains an endpoint or $S$ contains two consecutive vertices in $S$. Hence  the set $\{2,3\}$ is a zero forcing set, but is not adjacent to any minimum zero forcing set in $\ztar(P_n)$. Thus $\zbar(G)=2$.  It is easy to see that $\ulzo(P_n) =3$, so $z_0(P_n)=\ulzo(P_n)=3$.
\end{ex}

For TAR reconfiguration of the $X$-set parameters domination and power domination, examples are known  such that $\xxo(G)$ exceeds the lower bound $\xbar(G)+1$ (see \cite{HS14} and \cite{PD-recon}). Naturally, such examples are specific to the $X$-set parameter being studied, and knowing examples for one $X$-set parameter does not generally help construct such examples for other $X$-set parameters.   Next we construct a family of examples such that $\zbar(G)+r\le \zzo(G)<\min\{\zbar(G)+\Z(G),|V(G)|\}$ for $r\ge 2$.  Software such as \cite{sage:z-recon} is a useful tool for finding examples (and found $H(2)$, which we then generalized).

Define $H(r)$ to be the graph with $2r+4$ vertices such that both sets $V_1=\{1,\dots,r+2\}$ and $V_2=\{r+3,\dots,2r+4\}$ form cliques and there is a matching between the vertices $\{1,\dots,r\}$ and $\{r+3,\dots,2r+2\}$. The graphs $H(2)$ and $\ztar(H(2))$ are shown in Figure \ref{fig:TARdisconnect}.  

\begin{figure}[h!]
	\begin{center}
		\begin{tikzpicture}[scale=.65,every node/.style={draw,shape=circle,outer sep=2pt,inner sep=1pt,minimum size=.2cm}]   
	
	\node[fill=none]  (1) at (1,1) {\small{1}};
	\node[fill=none]  (2) at (1,-1) {\small{2}};
	\node[fill=none]  (3) at (3,-1) {\small{3}};
	\node[fill=none]  (4) at (3,1) {\small{4}};
	\node[fill=none]  (5) at (-1,1) {\small{5}};
	\node[fill=none]  (6) at (-1,-1) {\small{6}};
	\node[fill=none]  (8) at (-3,-1) {\small{8}};
	\node[fill=none]  (7) at (-3,1) {\small{7}};
	\draw[thick] (3)--(1)--(4)--(3)--(2)--(6)--(8)--(7)--(5)--(8);
	\draw[thick] (7)--(6)--(5)--(1)--(2)--(4);
	\end{tikzpicture}
		  \scalebox{.8}{	\begin{tikzpicture}[scale=.65,every node/.style={draw,shape=circle,outer sep=2pt,inner sep=1pt,minimum
					size=.2cm}]
				
				% Drawing the vertices
				\node[fill=black]  (35678) at (0, 7) {};
				\node[draw=none] at (0, 7.5){\small{$\{3,5,6,7,8\}$}};
				\node[fill=black]  (3567) at (2.5, 4) {};
				\node[draw=none, rotate=-70] at (3, 4){\small{$\{3,5,6,7\}$}};
				\node[fill=black]  (34567) at (2.5, 0) {};
				\node[draw=none, rotate=-90] at (2.9, 0){\small{$\{3,4,5,6,7\}$}};
				\node[fill=black]  (4567) at (2.5, -4) {};
				\node[draw=none, rotate=250] at (3, -4){\small{$\{4,5,6,7\}$}};
				\node[fill=black]  (45678) at (0, -7) {};
				\node[draw=none] at (0, -7.5){\small{$\{4,5,6,7,8\}$}};
				
				\node[fill=black]  (4568) at (-2.5, -4) {};
				\node[draw=none, rotate=110] at (-3, -4){\small{$\{4,5,6,8\}$}};
					\node[fill=black]  (34568) at (-2.5, 0) {};
				\node[draw=none,rotate=90] at (-2.9, 0){\small{$\{3,4,5,6,8\}$}};
				\node[fill=black]  (3568) at (-2.5, 4) {};
				\node[draw=none, rotate=70] at (-3, 4){\small{$\{3,5,6,8\}$}};
				
				\node[fill=black]  (3678) at (4, 4.5) {};
				\node[draw=none] at (5.3, 4.5){\small{$\{3,6,7,8\}$}};
				\node[fill=black]  (34678) at (4, 0) {};
				\node[draw=none, rotate=-90] at (4.5, 0){\small{$\{3,4,6,7,8\}$}};
				\node[fill=black]  (4678) at (4, -4.5) {};
				\node[draw=none] at (5.3, -4.5){\small{$\{4,6,7,8\}$}};
				
				\node[fill=black]  (4578) at (-4, -4.5) {};
				\node[draw=none] at (-5.3, -4.5){\small{$\{4,5,7,8\}$}};
				\node[fill=black]  (34578) at (-4, 0) {};
				\node[draw=none, rotate=90] at (-4.5, 0){\small{$\{3,4,5,7,8\}$}};
				\node[fill=black]  (3578) at (-4, 4.5) {};
				\node[draw=none] at (-5.3, 4.5){\small{$\{3,5,7,8\}$}};
			
			\node[fill=black]  (13678) at (3, 6.5) {};
			\node[draw=none, rotate=-60] at (3.5, 6.5){\small{$\{1,3,6,7,8\}$}};
			\node[fill=black]  (23678) at (5, 2.5) {};
			\node[draw=none, rotate=-90] at (5.5, 2.5){\small{$\{2,3,6,7,8\}$}};
			
				\node[fill=black]  (14678) at (5, -2.5) {};
			\node[draw=none, rotate=-90] at (5.5, -2.5){\small{$\{1,4,6,7,8\}$}};
			\node[fill=black]  (24678) at (3, -6.5) {};
			\node[draw=none, rotate=240] at (3.5, -6.5){\small{$\{2,4,6,7,8\}$}};
				
				\node[fill=black]  (13578) at (-3, 6.5) {};
				\node[draw=none, rotate=60] at (-3.5, 6.5){\small{$\{1,3,5,7,8\}$}};
				\node[fill=black]  (23578) at (-5, 2.5) {};
				\node[draw=none, rotate=90] at (-5.5, 2.5){\small{$\{2,3,5,7,8\}$}};
				
				\node[fill=black]  (14578) at (-5, -2.5) {};
				\node[draw=none, rotate=90] at (-5.5, -2.5){\small{$\{1,4,5,7,8\}$}};
				\node[fill=black]  (24578) at (-3, -6.5) {};
				\node[draw=none, rotate=120] at (-3.5, -6.5){\small{$\{2,4,5,7,8\}$}};
				
					\node[fill=black]  (13568) at (0, 5) {};
				\node[draw=none] at (0.2, 5.5){\small{$\{1,3,5,6,8\}$}};
				\node[fill=black]  (23568) at (0, 2) {};
				\node[draw=none] at (0.2, 2.5){\small{$\{2,3,5,6,8\}$}};
				
				\node[fill=black]  (13567) at (0, 3.5) {};
				\node[draw=none] at (0.2, 4){\small{$\{1,3,5,6,7\}$}};
				\node[fill=black]  (23567) at (0, 0.5) {};
				\node[draw=none] at (0.2, 1){\small{$\{2,3,5,6,7\}$}};
				
				\node[fill=black]  (14568) at (0, -0.5) {};
				\node[draw=none] at (0.2, -1){\small{$\{1,4,5,6,8\}$}};
				\node[fill=black]  (24568) at (0, -3.5) {};
				\node[draw=none] at (0.2, -4){\small{$\{2,4,5,6,8\}$}};
				
				\node[fill=black]  (14567) at (0, -2) {};
				\node[draw=none] at (0.2, -2.5){\small{$\{1,4,5,6,7\}$}};
				\node[fill=black]  (24567) at (0, -5) {};
				\node[draw=none] at (0.2, -5.5){\small{$\{2,4,5,6,7\}$}};
				
				%Drawing the thick edges
				\draw[thick] (35678)--(3567)--(34567)--(4567)--(45678)--(4568)--(34568)--(3568)--(35678)--(3678)--(34678)--(4678)--(45678)--(4578)--(34578)--(3578)--(35678);
				\draw[thick] (13678)--(3678)--(23678);
				\draw[thick] (14678)--(4678)--(24678);
				\draw[thick] (13578)--(3578)--(23578);
				\draw[thick] (14578)--(4578)--(24578);
				\draw[thick] (13568)--(3568)--(23568);
				\draw[thick] (14568)--(4568)--(24568);	
				\draw[thick] (14567)--(4567)--(24567);	
				\draw[thick] (13567)--(3567)--(23567);		
				
%drawing the other copy_new version
	% Drawing the vertices
\node[fill=black]  (12347) at (14, 7) {};
\node[draw=none] at (14, 7.5){\small{$\{1,2,3,4,7\}$}};
\node[fill=black]  (1237) at (16.5, 4) {};
\node[draw=none, rotate=-70] at (17, 4){\small{$\{1,2,3,7\}$}};
\node[fill=black]  (12378) at (16.5, 0) {};
\node[draw=none, rotate=-90] at (16.9, 0){\small{$\{1,2,3,7,8\}$}};
\node[fill=black]  (1238) at (16.5, -4) {};
\node[draw=none, rotate=250] at (17, -4){\small{$\{1,2,3,8\}$}};
\node[fill=black]  (12348) at (14, -7) {};
\node[draw=none] at (14, -7.5){\small{$\{1,2,3,4,8\}$}};

\node[fill=black]  (1248) at (11.5, -4) {};
\node[draw=none, rotate=110] at (11, -4){\small{$\{1,2,4,8\}$}};
\node[fill=black]  (12478) at (11.5, 0) {};
\node[draw=none,rotate=90] at (11.1, 0){\small{$\{1,2,4,7,8\}$}};
\node[fill=black]  (1247) at (11.5, 4) {};
\node[draw=none, rotate=70] at (11, 4){\small{$\{1,2,4,7\}$}};

\node[fill=black]  (2347) at (18, 4.5) {};
\node[draw=none] at (19.3, 4.5){\small{$\{2,3,4,7\}$}};
\node[fill=black]  (23478) at (18, 0) {};
\node[draw=none, rotate=-90] at (18.5, 0){\small{$\{2,3,4,7,8\}$}};
\node[fill=black]  (2348) at (18, -4.5) {};
\node[draw=none] at (19.3, -4.5){\small{$\{2,3,4,8\}$}};

\node[fill=black]  (1348) at (10, -4.5) {};
\node[draw=none] at (8.7, -4.5){\small{$\{1,3,4,8\}$}};
\node[fill=black]  (13478) at (10, 0) {};
\node[draw=none, rotate=90] at (9.5, 0){\small{$\{1,3,4,7,8\}$}};
\node[fill=black]  (1347) at (10, 4.5) {};
\node[draw=none] at (8.7, 4.5){\small{$\{1,3,4,7\}$}};

\node[fill=black]  (23457) at (17, 6.5) {};
\node[draw=none, rotate=-60] at (17.5, 6.5){\small{$\{2,3,4,5,7\}$}};
\node[fill=black]  (23467) at (19, 2.5) {};
\node[draw=none, rotate=-90] at (19.5, 2.5){\small{$\{2,3,4,6,7\}$}};

\node[fill=black]  (23458) at (19, -2.5) {};
\node[draw=none, rotate=-90] at (19.5, -2.5){\small{$\{2,3,4,5,8\}$}};
\node[fill=black]  (23468) at (17, -6.5) {};
\node[draw=none, rotate=240] at (17.5, -6.5){\small{$\{2,3,4,6,8\}$}};

\node[fill=black]  (13457) at (11, 6.5) {};
\node[draw=none, rotate=60] at (10.5, 6.5){\small{$\{1,3,4,5,7\}$}};
\node[fill=black]  (13467) at (9, 2.5) {};
\node[draw=none, rotate=90] at (8.5, 2.5){\small{$\{1,3,4,6,7\}$}};

\node[fill=black]  (13458) at (9, -2.5) {};
\node[draw=none, rotate=90] at (8.5, -2.5){\small{$\{1,3,4,5,8\}$}};
\node[fill=black]  (13468) at (11, -6.5) {};
\node[draw=none, rotate=120] at (10.5, -6.5){\small{$\{1,3,4,6,8\}$}};

\node[fill=black]  (12457) at (14, 5) {};
\node[draw=none] at (14.2, 5.5){\small{$\{1,2,4,5,7\}$}};
\node[fill=black]  (12467) at (14, 2) {};
\node[draw=none] at (14.2, 2.5){\small{$\{1,2,4,6,7\}$}};

\node[fill=black]  (12357) at (14, 3.5) {};
\node[draw=none] at (14.2, 4){\small{$\{1,2,3,5,7\}$}};
\node[fill=black]  (12367) at (14, 0.5) {};
\node[draw=none] at (14.2, 1){\small{$\{1,2,3,6,7\}$}};
%here
\node[fill=black]  (12458) at (14, -0.5) {};
\node[draw=none] at (14.2, -1){\small{$\{1,2,4,5,8\}$}};
\node[fill=black]  (12468) at (14, -3.5) {};
\node[draw=none] at (14.2, -4){\small{$\{1,2,4,6,8\}$}};

\node[fill=black]  (12358) at (14, -2) {};
\node[draw=none] at (14.2, -2.5){\small{$\{1,2,3,5,8\}$}};
\node[fill=black]  (12368) at (14, -5) {};
\node[draw=none] at (14.2, -5.5){\small{$\{1,2,3,6,8\}$}};

%Drawing the thick edges
\draw[thick] (12347)--(1237)--(12378)--(1238)--(12348)--(1248)--(12478)--(1247)--(12347)--(2347)--(23478)--(2348)--(12348)--(1348)--(13478)--(1347)--(12347);
\draw[thick] (23457)--(2347)--(23467);
\draw[thick] (23458)--(2348)--(23468);
\draw[thick](13457)--(1347)--(13467);
\draw[thick] (13458)--(1348)--(13468);
\draw[thick] (12457)--(1247)--(12467);
\draw[thick] (12357)--(1237)--(12367);	

\draw[thick] (12458)--(1248)--(12468);	
\draw[thick] (12358)--(1238)--(12368);							
		\end{tikzpicture}}
\end{center}
    \caption{The graphs $H(2)$ (top) and $\ztar_5(H(2))$ (bottom)}  \label{fig:TARdisconnect}
\end{figure}

\begin{prop}\label{p:upper-con-z}
For $r\ge 2$, $\Z(H(r))=\zbar(H(r))=r+2$ and $\ulzo(H(r))=\zzo(H(r)) = 2r+2= \zbar(H(r)) + r$.
\end{prop}
\bpf   Observe that the vertices $r+1,r+2,2r+3$ and $2r+4$ have degree $r+1$ and all other vertices of $H(r)$ have degree $r+2$ so $\delta(H(r))=r+1$.  Let $U_1=\{r+1,r+2\}$ and $U_2=\{2r+3,2r+4\}$. If $v\in V_i$, then $|N_{H(r)}(v)\cap V_i|=r+1$.  Let $S\subset V(H(r))$ and let $S_i=S\cap V_i$ for $i=1,2$. Assuming that $|S_1|\ge |S_2|$, we show that $S$ is a zero forcing set of $H(r)$ if and only if $|S_1|\ge r+1$ and $|S_2\cap U_2|\ge 1$. Once this is established, $\Z(H(r))=\zbar(H(r))=r+2$ with a minimum zero forcing set $S$ having ($|S_1|= r+1$ and $|S_2|= 1$) or ($|S_2|= r+1$ and $|S_1|= 1$).  Furthermore, $\ztar_{2r+1}(H(r))$ has 2 components, one containing zero forcing sets $S$ with $|S_1|>|S_2|$ and the other with $|S_2|>|S_1|$, whereas  $\ztar_{2r+2}(H(r))$ is connected.  

Suppose first that $|S_1|=r+1$, $|S_2|=1$, and $S_2\subset U_2$ (the argument where $|S_2|=r+1$ is similar).  There must be a vertex $u\in U_1\cap S_1$, and $u$ can force the one white vertex of $S_1$.  Then $i\to r+2+i$ for $i=1,\dots,r$.  Finally any blue vertex of $V_2$ can force the one remaining white vertex in $V_2$.  Thus $S$ is a zero forcing set. 

Now suppose $S$ is a zero forcing set.  Let $v$ be the vertex that performs the first force, and without loss of generality $v\in V_1$.  Since $v$ and all but one of its neighbors must be in $S$, $S$ contains at least $r+1$ of the vertices in $S$.  Since at most $r$ vertices of $V_2$  can be forced by vertices in $V_1$, $S$ must contain a vertex of $V_2$ that has no neighbor in $V_1$, i.e., $S$ must contain at least one vertex of $U_2$. 
\epf

A computer search on graphs up to 8 vertices (with no isolated vertices) found exactly 2 graphs $G$ such that $\zzo(G)>\ulzo(G)$. For example, the graph $H_2$ in Figure \ref{1stconn} has $\zzo(H_2) = 7$ and $\ulzo(H_2) = 5$ (because there are minimal zero forcing sets of orders 4 and 6 but not 5, and $\ztar_5(H_2)$ is  connected, see \cite{sage:z-recon}).

\begin{figure}[h!]
\centering 

	\begin{tikzpicture}[scale=.7,every node/.style={draw,shape=circle,outer sep=2pt,inner sep=1pt,minimum
			size=.2cm}]
		
		\node[fill=none]  (0) at (-1.5,2) {};
		\node[fill=none]  (1) at (0,2.5) {};
		\node[fill=none]  (2) at (1.5,2) {};
		\node[fill=none]  (3) at (-1,0.5) {};
		\node[fill=none]  (4) at (0,0) {};
		\node[fill=none]  (5) at (1,0.5) {};
		\node[fill=none]  (6) at (-1.3,-1.6) {};
		\node[fill=none]  (7) at (1.3,-1.6) {};
		
		\node[draw=none] at (-1.3, -2.1){$\small{u_1}$};
		\node[draw=none] at (1.3,-2.1){$\small{u_2}$};
		
		%Drawing the thick edges
		\draw[thick] (0)--(1)--(2)--(5)--(4)--(3)--(0)--(5)--(1)--(4)--(7)--(3)--(6)--(4)--(0);
		\draw[thick] (6)--(4)--(2)--(3)--(1);
	    \draw[thick](7)--(5)--(6);
		\end{tikzpicture}
		
    \caption{A graph $H_2$ satisfying $\zzo(H_2)>\ulzo(H_2)$ }  \label{1stconn}
\end{figure}

Next we show that the graph $H_2$ can be used to create an infinite family of graphs $H_r$ such that $\zzo(H_r)>\ulzo(H_r)$.  Observe that  $N_{H_2}(u_2)=N_{H_2}(u_1)$.   Construct $H_r$  from $H_2$ by adding vertices $u_3,\dots,u_r$ with  $N_{H_r}(u_k)=N_{H_r}(u_1)$ for $k=3,\dots,r$.  In general, vertices $u_1$ and $u_2$ in a graph $G$ are called \emph{twins} if  $N_G(u_1) = N_G(u_2)$ (what we define as twins are often called \emph{independent twins}). A \emph{set of twins} is a set  $\{u_1,\dots,u_k\}\subseteq V(G)$ such that $u_i$ and $u_j$ are twins  for all $1\le i<j\le k$.

\begin{obs} \label{obs:twin}
If $G$ is a graph with a set of twins $\{u_1,\dots ,u_r\}$, then any zero forcing set of $G$ must contain at least $r-1$ of the vertices $\{ u_1,\dots ,u_r \} $. \end{obs} 

\begin{prop}\label{p:ztwin} 
Let $G$ be a graph that has a set of twins $T=\{u_1,\dots ,u_r\}$ with $r\ge 3$,  and let $G_i = G-u_i$. If $S_i$ is a zero forcing set of $G_i$, then $S=S_i\cup\{u_i\}$ is a zero forcing set of $G$.  If $S$ is a zero forcing set of $G$ and $ u_i\in S$, then  $S_i=S\setminus\{u_i\}$ is a zero forcing set of $G_i$.  Thus there is a bijection between zero forcing sets of $G_i$ and zero forcing sets of $G$ that contain $u_i$, and a zero forcing set 
 $S_i$ of $G_i$ is minimal if and only if $S=S_i\cup\{u_i\}$ is a minimal zero forcing set of $G$.
\end{prop}
 \bpf It is immediate that $S_i$ being a zero forcing set of $G_i$ implies $S=S_i\cup\{ u_i\}$ is a zero forcing set of $G$ (without any assumption about twins). Let $S$ be a zero forcing set of $G$. Since $r\ge 3$, $S$ must contain at least two vertices in $T$, say $u_i$ and $u_j$. Choose a set of forces $\clf$ that color every vertex of $G$ blue.  We show that at least one of $u_i$ and $u_j$ does not perform a force. Neither $u_i$ nor $u_j$ can perform a force until all but one of their common set of neighbors are blue.  So if $u_i$ forces the one white neighbor blue, then $u_j$ cannot perform a force and if $u_j$ forces, then $u_i$ can't (it is possible neither forces). Thus we may assume $u_i$ does not perform a force. Then $S_i=S\setminus\{u_i\}$ is a zero forcing set of $G_i$ using the set of forces $\clf$. 
 The last sentence is immediate. \epf

 The technique of removing a nonforcing vertex $v$ from a zero forcing set $S$ of $G$  to obtain a zero forcing set $S\setminus\{v\}$ is well known (see, for example, the proof of Theorem 2.7 in \cite{EHHLR12}).  When there are only two twins, it is not true that $S_i=S\setminus\{u_i\}$ must be a zero forcing set of $G_i$, because it is possible  that only one of $u_1$ and $u_2$ is in $S$ and that vertex is required to perform a force. This is illustrated in the next example.

\begin{ex}
Consider the \emph{double star} $DS(2,2)$ constructed by adding an edge between the centers of two disjoint copies of $K_{1,2}$.  If the leaves of one star are $u_1$ and $u_2$, then $\Z(DS(2,2))=2=\Z(DS(2,2)-u_1)$ (see \cite{sage:z-recon}).
\end{ex} 

\begin{lem}\label{l:twin-conn}
Let $G$ be a graph that has a set of twins $T=\{u_1,\dots ,u_r\}$ with $r\ge 3$,  and let $G_i = G-u_i$.  If $\ztar_{k}(G_i)$ is connected for some $i$, then $\ztar_{k+1}(G)$ is connected.
\end{lem}

\begin{proof} The graphs  $G_i$ are isomorphic and hence the $\ztar_{k}(G_i)$ are isomorphic as well.  Thus if $\ztar_{k}(G_i)$ is connected for some $i$  then $\ztar_{k}(G_i)$ is connected for all $i$.

For  $S\subseteq V(G)$ with $u_i\in S$, let $S_i=S\setminus \{u_i\}$. 
Since $r\ge 3$,  any zero forcing set $S$ of $G$ must contain at least two vertices in $T$, say $u_i$ and $u_j$.   By Proposition \ref{p:ztwin},  a set  $S\subseteq V(G) $  that contains $u_i, u_j$  is a zero forcing set of $G$  if and only if  $S \setminus {u_i}$ is a zero forcing set of $G_i$ and  $S \setminus {u_j}$ is a zero forcing set of $G_j$.  

Let $S, S'$ be  two zero forcing sets  of $G$ of size $k+1$ or less.  Since $r\ge 3$  and each can omit at most one  vertex in $T$, their intersection must contain at least one  $u_i$. Then $S\setminus {u_i}, S'\setminus {u_i} $ are zero forcing sets for $G_i$. Since by assumption $\ztar_{k}(G_i)$ is connected, that means that there is a path between $S\setminus {u_i}$ and  $S'\setminus {u_i} $ in $\ztar_{k}(G_i)$ and hence a path between $S, S'$ in $\ztar_{k+1}(G)$.
\end{proof}

\begin{prop}  For $r\ge 2$, $\Z(H_r)=r+2$, $\ulzo(H_r)=r+3$, and $\zzo(H_r)=r+5$.
\end{prop}
\bpf The proof is by induction.  The base case is $H=H_2$, which has minimal zero forcing sets of sixes 4 and 6 and has $\ztar_5(H_2)$ is connected.  Apply Proposition \ref{p:ztwin} to show that $H_r$ has minimal zero forcing sets of sizes $r+2$ and $r+4$.  Apply Lemma \ref{l:twin-conn} to show $\ztar_{r+3}(H_r)$ is connected.
\epf

\section*{Acknowledgements}

%The research of N.~Bong, J.~Carlson, and B.~Curtis was partially supported by NSF grant 1929334. 

 The authors thank the American Institute of Mathematics (AIM) and the National Science Foundation (NSF) for support of this research.  The research of B. Curtis was also partially supported by NSF grant 1839918.

%This research began in the AIM Research Community \emph{Inverse eigenvalue problems for graphs} with  support from NSF grant 1929334. The authors thank the American Institute of Mathematics (AIM) and the national Science Foundation (NSF).

\end{document}